\documentclass[oneside]{amsart}

\usepackage{amssymb}

\usepackage{graphicx}

\usepackage[latin1]{inputenc}
\usepackage{amssymb,amsmath}
\usepackage{verbatim}
\usepackage{color}
\usepackage{enumerate}
\usepackage{url}
\usepackage{geometry}
\usepackage{mathtools}
\input{xypic}

\newtheorem{theorem}{Theorem}[section]
\newtheorem{lemma}[theorem]{Lemma}
\newtheorem{proposition}[theorem]{Proposition}
\newtheorem{conjecture}[theorem]{Conjecture}

\theoremstyle{definition}

\newtheorem{ass}[theorem]{Assumption}

\theoremstyle{remark}
\newtheorem{remark}[theorem]{Remark}

\numberwithin{equation}{section}

\newcommand\Z{\ensuremath{\mathbb Z}}\newcommand\A{\ensuremath{\mathbb A}}
\newcommand\Q{\ensuremath{\mathbb Q}}\newcommand\R{\ensuremath{\mathbb R}}
\newcommand\C{\ensuremath{\mathbb C}}
\newcommand\Qb{{\overline\Q}}

\newcommand\Aut{\operatorname{Aut}}

\newcommand\coker{\operatorname{coker}}

\newcommand\Div{\operatorname{Div}}
\newcommand\End{\operatorname{End}}

\newcommand\Gal{\operatorname{Gal}}
\newcommand\GL{\operatorname{GL}}
\newcommand\Hom{\operatorname{Hom}}

\newcommand\M{\operatorname{M}}

\newcommand\ord{\operatorname{ord}}

\newcommand\PGL{\operatorname{PGL}}
\newcommand\PSL{\operatorname{PSL}}

\newcommand\SL{\operatorname{SL}}

\newcommand\Jac{\operatorname{Jac}}
\renewcommand{\cH}{\mathcal{H}}
\newcommand{\fp}{{\mathfrak{p}}}

\newcommand{\fl}{{\mathfrak{l}}}

\newcommand{\fL}{{\mathfrak{L}}}
\newcommand{\fn}{{\mathfrak{n}}}
\newcommand{\Qbar}{{\overline{\Q}}}

\newcommand{\fd}{{\mathfrak{d}}}

\newcommand{\cU}{{\mathcal{U}}}
\newcommand{\HH}{\mathcal{H}_3}
\newcommand{\Meas}{\mathrm{Meas}}
\newcommand{\codim}{\mathrm{codim}}

\def\cO{{\mathcal O}}
\newcommand\acc[2]{\ensuremath{{}^{#1}\hskip-0.7ex{#2}}}

\newcommand{\mtx}[4]{\begin{pmatrix*}[r]#1&#2\\#3&#4\end{pmatrix*}}
\newcommand{\mat}[1]{\begin{pmatrix*}[r]#1\end{pmatrix*}}

\newcommand{\smtx}[4]{\left(\begin{smallmatrix}#1&#2\\#3&#4\end{smallmatrix}\right)}

\def\M{\operatorname{M}}

\def\P{\mathbb P}

\newcommand{\comp}{\begin{picture}(6,5)(-3,-2)\put(0,1){\circle{2}} \end{picture}}\def\circ{\comp}
\def\fN{\mathfrak N}
\def\fm{\mathfrak m}

\def\fn{\mathfrak n}

\newcommand{\ra}{\rightarrow}

\newcommand{\lra}{\longrightarrow}

\def\Xint#1{\mathchoice
{\XXint\displaystyle\textstyle{#1}}%
{\XXint\textstyle\scriptstyle{#1}}%
{\XXint\scriptstyle\scriptscriptstyle{#1}}%
{\XXint\scriptscriptstyle\scriptscriptstyle{#1}}%
\!\int}
\def\XXint#1#2#3{{\setbox0=\hbox{$#1{#2#3}{\int}$}
\vcenter{\hbox{$#2#3$}}\kern-.5\wd0}}

\begin{document}

\title[Periods of modular $\GL_2$-type abelian varieties and $p$-adic integration]{Periods of modular $\GL_2$-type abelian varieties\\ and $p$-adic integration}

\author{Xavier Guitart}
\address{Departament d'Algebra i Geometria, Universitat de Barcelona, Catalonia}
\curraddr{}
\email{xevi.guitart@gmail.com}

\author{Marc Masdeu}
\address{Mathematics Institute, University of Warwick, United Kingdom}
\curraddr{}
\email{M.Masdeu@warwick.ac.uk}


\date{\today}

\dedicatory{}

\begin{abstract}
Let $F$ be a number field and $\fN\subset \cO_F$ an integral ideal. Let $f$ be a modular newform over $F$ of level $\Gamma_0(\fN)$ with rational Fourier coefficients. Under certain additional conditions, \cite{GMS2} constructs a $p$-adic lattice which is conjectured to be the Tate lattice of an elliptic curve $E_f$ whose $L$-function equals that of $f$. The aim of this note is to generalize this construction when the Hecke eigenvalues of $f$ generate a number field of degree $d\geq 1$, in which case the geometric object associated to $f$ is expected to be, in general, an abelian variety $A_f$ of dimension $d$. We also provide numerical evidence supporting the conjectural construction in the case of abelian surfaces.
\end{abstract}

\maketitle

\section{Introduction}\label{section 1}
Let $F$ be a number field and $\fN\subset \cO_F$ an integral ideal. Let $f$ be a modular newform over $F$ of level $\Gamma_0(\fN)$ with rational Fourier coefficients. Under certain additional conditions, \cite{GMS2} constructs a $p$-adic lattice which is conjectured to be the Tate lattice of an elliptic curve $E_f$ whose $L$-function equals that of $f$. The aim of this note is to generalize this construction when the Hecke eigenvalues of $f$ generate a number field of degree $d\geq 1$, in which case the geometric object associated to $f$ is expected to be, in general, an abelian variety $A_f$ of dimension $d$. That is, we attach to $f$ a lattice that conjecturally uniformizes $A_f$ over $\C_p$. In the particular case of abelian surfaces, we provide numerical verifications of the conjecture, as well as some further arithmetic applications.

To put this construction into perspective, let us first briefly recall the situation for classical modular forms over $\Q$. Let $f=\sum_{n\geq 1}a_nq^n$ momentarily denote a weight two newform for the congruence subgroup $\Gamma_0(N)\subset \SL_2(\Z)$, and let $K_f=\Q(\{a_n\}_{n\geq 1})$ be the field generated by the Hecke eigenvalues of $f$. A construction of Eichler and Shimura (cf. \cite[\S 7.5]{Sh-book}) associates to $f$ an abelian variety $A_f/\Q$ of dimension $d=[K_f\colon\Q]$ and conductor $N^d$ that satisfies the equality of $L$-functions
\begin{align}\label{eq: equalityLfunctions}
  L(A_f/\Q,s)= \prod_{i=1}^d L(\acc{\sigma_i} f,s),
\end{align}
where the $\sigma_i$ run over the embeddings of $K_f$ into $\Qb$, and $\acc\sigma f= \sum_{n\geq 1} \sigma(a_n)q^n$. Moreover, the algebra $\Q\otimes\End(A_f)$ of endomorphisms defined over $\Q$ is isomorphic to $K_f$. 

Recall that in this setting $\Gamma_0(N)$ acts on the complex upper half plane $\cH$, and the compactification of the quotient space $\Gamma_0(N)\backslash\cH$ are the $\C$-points of the modular curve $X_0(N)/\Q$. The Eichler--Shimura construction realizes $A_f$ quite explicitly as a simple factor of the Jacobian of $X_0(N)$, and in a manner that is amenable for numerical calculations. In particular, the lattice uniformizing $A_{f}\otimes \C$ can be given as
\begin{align}\label{eq: lambda_f}
  \Lambda_f = \left\{ \Big(\int_\gamma \omega_1,\dots, \int_{\gamma} \omega_d \Big) \colon \gamma\in H_1(X_0(N),\Z)\right\}\subset \C^d,
\end{align}
where $\omega_i$ denotes the differential form corresponding to $\acc{\sigma_i}f(z)dz$. This lattice can be efficiently computed by means of the modular symbols method \cite{cremona-book}, \cite{stein-book} and, in the case of dimension one, namely when $A_f$ is an elliptic curve, is the base for producing exhaustive tables of elliptic curves up to a certain conductor (currently, the Cremona database (\cite{cremonatables}) contain all elliptic curves up to conductor 380,000). There is also extensive literature on higher dimensional computations, specially in dimension $2$ (see, e.g., \cite{wang}, \cite{flynn-et-al}, \cite{GGG}, \cite{GGR}). When $A_f$ admits a principal polarization, some of these works also provide methods to compute equations of genus $2$ curves whose Jacobian is isogenous to $A_f$.
 
When $f$ is a modular form over a number field $F$ other than $\Q$, an abelian variety $A_f$ satisfying \eqref{eq: equalityLfunctions} is also expected to exist in general\footnote{Note that condition \eqref{eq: equalityLfunctions} characterizes $A_f$ up to isogeny; we will abuse notation and denote by $A_f$ any variety satisfying \eqref{eq: equalityLfunctions}.}. This is only known in some cases for totally real fields $F$, but if $F$ has some complex place the conjecture is completely open. We will state this conjecture more precisely in Section \ref{section 2} below, but in order to describe the contents of the article let us give a brief overview. 

We will only treat the case where $F$ has narrow class number one, so we make this assumption from now on. In this setting, the newform $f$ over $F$ can be identified with a harmonic differential form on the orbifold $$\Gamma_0(\fN)\backslash (\cH^r\times \HH^s),$$ where $r$ (resp. $s$) is the number of real (resp. complex) places of $F$ and $\HH=\C\times \R_{>0}$ denotes the hyperbolic upper half space. The only situation where the Eichler--Shimura construction over $\Q$ generalizes satisfactorily is when $F$ is totally real and $f$ admits a Jacquet--Langlands transfer to a modular form $f^B$ for some arithmetic subgroup $\Gamma_0^B(\fm)\subset B$, with $B/F$ a quaternion algebra that ramifies at all infinite places but one. The form $f^B$ corresponds in this case to a holomorphic differential form on $\Gamma_0^B(\fm)\backslash\cH$, which are the $\C$-points of a Shimura curve $X_B$ defined over $F$. Then $A_f$, as a variety over $F$, can be constructed as a quotient of the Jacobian of $ X_B$. Over $\C$, one can describe its period lattice as in \eqref{eq: lambda_f}; that is, as periods of the differentials attached to $f^B$ and its conjugates. 

When $F$ is totally real and $f$ does not admit a suitable Jacquet--Langlands lift (the simplest case where this happens is when $[F:\Q]=2$ and $\fN = \cO_F$) no construction of $A_f$ is known. However, a conjecture of Oda \cite{oda} describes the complex period lattice of $A_f$ in terms of certain periods of the Hilbert modular form $f$, and is the basis of the algorithm introduced in \cite{De} to compute equations of $A_f$ in the $1$-dimensional case. More recently, abelian surfaces of trivial conductor over real quadratic fields attached to Hilbert modular forms have been computed using this approach in \cite{De-Ku}, thus providing numerical evidence also for Oda's conjecture in higher dimension.

The situation seems to be even more mysterious when $F$ is not totally real. Indeed, in this case there is no apparent connection to algebraic geometry because $\Gamma_0(\fN)\backslash(\cH^r\times \HH^s)$ has no complex structure if $s>0$. Since no algebraic variety seems to present itself as a natural candidate to give rise to $A_f$, no geometric construction of $A_f$ is known in this context. In fact, to the best of our knowledge, such a construction has not even been conjectured.

In the one-dimensional case, there is a huge amount of experimental evidence supporting the existence of elliptic curves attached to modular forms over non totally real fields (see \cite{grunewald}, \cite{cremona}, \cite{whitley}, \cite{gunnells_5}, \cite{gunnells_23}, \cite{jones}, \cite{ellipticcurvesearch}). In this setting, \cite{GMS} contains two conjectural analytic constructions of the period lattice of the elliptic curve $A_f$. The first construction concerns the complex lattice of $A_f$, and is a generalization of Oda's conjecture to number fields having at least one real place, in the spirit of the work of Darmon--Logan \cite{darmon-logan} and Gartner \cite{gartner}. The second, which builds on constructions of Darmon \cite{darmon-integration} and Greenberg \cite{Gr} in the context of Stark--Heegner points, is for the $p$-adic Tate lattice of $A_f$. Numerical evidence for the $p$-adic construction, as well as an algorithm to compute the equation of the elliptic curve $A_f$ from its $p$-adic periods were presented in \cite{GMS2}.

The aim of the present note is to generalize the construction of the $p$-adic lattice of $A_f$ to the case where $d=\dim A_f  > 1$. In the case of $F=\Q$ this generalization was obtained by Dasgupta in~\cite{dasgupta-shpoints} and by Longo--Rotger--Vigni in~\cite{LRV} (see also the work~\cite{rotger-seveso} for a generalization to modular forms of higher weight). The construction that we present is valid for number fields of arbitrary degree and signature, under the assumption that there exists a prime $\fp$ dividing exactly the level of the modular form $f$. We show that in that case the abelian variety $A_f$, if it exists, admits a $\fp$-adic uniformization: there exists a lattice $\Lambda_f\subset (\C_p^\times)^d$ such that $$A_f(\C_p)\simeq (\C_p^\times)^d/\Lambda_f,$$ where $p=\fp\cap \Z$ and $\C_p$ is the completion of an algebraic closure of $\Q_p$. The main construction of the present article associates to the modular form $f$ a certain $p$-adic lattice $\Lambda_f'$ defined by means of a $p$-adic integration pairing, which can be regarded as a $p$-adic analog of \eqref{eq: lambda_f}. We then conjecture that $(\C_p^\times)^d/\Lambda_f'$ and $(\C_p^\times)^d/\Lambda_f$ are isogenous abelian varieties, thus providing a conjectural $p$-adic analytic construction of the abelian variety $A_f$ satisfying \eqref{eq: equalityLfunctions}.

 An alternative way of phrasing the conjecture above is as asserting the equality of the $\mathcal{L}$-invariant of $\Lambda_f'$ and the $\mathcal{L}$-invariant of $A_f$. Indeed, equality of the $\mathcal{L}$-invariants of two $p$-adic lattices of rank $d$ with an action of a number field of degree $d$ is equivalent to the corresponding rigid analytic tori being isogenous. The $\mathcal{L}$-invariant of $\Lambda_f'$ can then be regarded as an automorphic Darmon-style $\mathcal{L}$-invariant. The theme of relating automorphic $\mathcal{L}$-invariants to other $\mathcal{L}$-invariants, such as the geometric one, is also relevant for its connections with $p$-adic $L$-functions. The reader can consult \cite{darmon-integration}, \cite{Das}, \cite{greenberg-dasgupta}, \cite{LRV}, \cite{seveso}, or \cite{spiess} for results and conjecutures in this direction.

The paper is organized as follows. Section \ref{section 2} contains some background material on the relation between modular forms and certain cohomology classes, as well as to the abelian varieties conjecturally attached to them. We also record some results on $p$-adic uniformization and $p$-adic lattices. We describe the construction of the lattice that conjecturally uniformizes $A_f$ in Section \ref{section 3}, which in fact is a natural generalization of the elliptic curve case of \cite{darmon-integration}, \cite{Gr}, and \cite{GMS2}. One of the main points of the present note is to provide numerical evidence supporting the conjecture in the case where $\dim A_f = 2$, and this is the main content of Section \ref{section 4}: in \S\ref{Algorithms for computing the $p$-adic lattice} we discuss the algorithms that we used for the calculations (which work under the additional assumption that the number field has at most one complex place); in \S\ref{subsection: p-adic periods of genus two curves} we recall the formulas of Teitelbaum \cite{Tei} to compute $p$-adic lattices of genus two curves; and in \S\ref{subsection: a numerical verification} we report on the explicit computation of the $p$-adic lattice $\Lambda_f'$ of a modular form over a number field $F$ of signature $(1,1)$ whose Hecke eigenvalues generate a quadratic number field. We check (up to the working precision of $50$ $p$-adic digits) that this lattice is isogenous  with the $p$-adic lattice of a genus two curve whose Jacobian is $A_f$. Finally, in Sections \ref{section 5} and \ref{section 6}, we give two additional applications of our construction. The first one is to computing equations of genus two curves whose Jacobian is the variety $A_f$ attached to $f$, by means of the explicit uniformization formulas for genus two curves of \cite{Tei}.  The second is to the (conjectural) computation of the $p$-adic $L$-invariant of $A_f$. 

\subsection*{Notation} If $F$ is a number field we denote by $\cO_F$ its ring of integers, and we say that $F$ is of signature $(r,s)$ if it has $r$ real places and $s$ complex places. For an abelian variety $A$ defined over $F$, we denote by $\End(A)$ the endomorphisms of $A$ defined over $F$. For an extension $L/F$,  $A_L$ denotes the base change $A\times_{\operatorname{Spec} K}\operatorname{Spec} L$; consequently $\End(A_L)$ stands for the endomorphisms of $A$ defined over $L$.

\subsubsection*{Acknowledgments} We wish to thank Lassina Dembele, Ariel Pacetti, Haluk Sengun, John Voight, and Xavier Xarles for feedback and helpful conversations during this project. Masdeu thanks the Number Theory group of the University of Warwick for provinding an outstanding working environment, and Guitart is thankful to the Essen Seminar for Algebraic Geometry and Arithmetic for their hospitality during his stay. Guitart was supported by MTM2015-66716-P and MTM2015-63829, and Masdeu was supported by MSC--IF--H2020--ExplicitDarmonProg. This project has received funding from the European Research Council (ERC) under the European
Union's Horizon 2020 research and innovation programme (grant agreement No 682152).

\section{Modular forms, cohomology classes,  and $p$-adic uniformization}\label{section 2}
We begin this section by formulating a precise version of the conjecture that associates an abelian variety to any modular form over a number field,  mainly following the presentation of \cite{Taylor}. We describe it not only for modular forms over $\GL_2$, but over arbitrary quaternion algebras over $F$. By the Jacquet--Langlands correspondence, the systems of Hecke eigenvalues on quaternion algebras already arise on the split algebra $\GL_2$, so one does not gain much from a theoretical point of view. However, as we will see in Section \ref{section 4}, for computational purposes it is sometimes helpful to transfer the problem to a non-split quaternion algebra. 

The complex upper half plane $\cH$ is endowed with an action of $\PSL_2(\R)$ by fractional linear transformations. Similarly, the hyperbolic upper half space $\HH$ is acted on by $\PSL_2(\C)$ as follows: if we let $\mathcal{Q}=\C \oplus j\C$ denote Hamilton's quaternions and we identify $\HH$ with  $\{x + j y\in \mathcal{Q}\colon y\in \R_{>0}\}$, then a matrix acts on  $z\in \HH$ by the formula
\begin{align*}
\smtx{a}{b}{c}{d}\cdot z = (az + b)\cdot(cz+d)^{-1}\ \ \text{(the multiplication is in $\mathcal{Q}$).}
\end{align*}
In the construction of Section \ref{section 3} we will consider modular forms over $F$ with the property that there is a prime dividing $\fN$ exactly, so we introduce this assumption in the levels of the modular forms considered in this section. In fact, let us assume from now on that $\fN$  admits a factorization into coprime ideals of the form
\begin{align*}
  \fN = \fp \fd \fm, \ \ \text{ with $\fp$ prime and $\fd$ square-free. }
\end{align*}
 The condition that there exists a prime $\fp$ dividing $\fN$ exactly is inherent to this kind of construction, and it was already present in the works of Darmon and Greenberg \cite{darmon-integration} \cite{Gr}. The condition that $\fd$ is square-free arises because we want to translate the problem to a quaternion algebra. Let $B/F$ be a quaternion algebra of discriminant $\fd$ and which is split at $n\leq r$ real places of $F$. We remark that $\fd$ is allowed to be trivial. If in addition $n=r$, then $B$ is simply the matrix algebra $M_2(F)$.

Let $R_0(\fp\fm)\subset R_0(\fm)$ be Eichler orders in $B$ of levels $\fp\fm$ and $\fm$ respectively. Let $R_0(\fp\fm)_1^\times$ and $R_0(\fm)_1^\times$ denote their group of norm $1$ units and put $\Gamma_0(\fp\fm)=R_0(\fp\fm)_1^\times/\{\pm 1\}$ and $\Gamma_0(\fm)=R_0(\fm)_1^\times/\{\pm 1\}$.  By fixing isomorphisms $B\otimes_\sigma \R\simeq \M_2(\R)$ for all real places $\sigma$ of $F$ at which $B$ is split and $B\otimes_\sigma \C\simeq \M_2(\C)$ for all complex places  of $F$, the group  $\Gamma_0(\fp\fm)$ acts on $\cH^n \times \HH^s$.  Results of Harder \cite{harder-87} allow to interpret modular forms for $\Gamma_0(\fp\fm)$ either as harmonic differential forms on the quotient $\Gamma_0(\fp\fm)\backslash( \cH^n \times \HH^s)$ or, equivalently, as cohomology classes in the Betti cohomology group $H^{n+s}(\Gamma_0(\fp\fm)\backslash (\cH^n \times \HH^s),\C)$. We will take from now on the latter point of view.

For the sake of simplicity, we shall assume that $\Gamma_0(\fp\fm)$ is torsion free. This implies, in particular, that for any abelian group $A$ we have canonical isomorphisms
\begin{align*}
  H^{n+s}(\Gamma_0(\fp\fm)\backslash \cH^n \times \HH^s,A) \simeq H^{n+s}(\Gamma_0(\fp\fm),A)\simeq H^{n+s}(\Gamma_0(\fp\fm),\Z)\otimes A,
\end{align*}
where the group on the left is Betti cohomology and the others represent group cohomology.

For any $B^\times/F^\times$-module $V$ and any ideal $\fn\subset \cO_F$, the cohomology groups $H^{i}(\Gamma_0(\fn),V)$ are endowed with the action of the Hecke operators defined by means of the formalism of double coset operators (see, for example, \cite[\S 1.1]{ash-stevens}): for every prime $\fl$ of $F$ not dividing $\fd$ there is an endomorphism 
\begin{align*}
T_\fl\colon H^{i}(\Gamma_0(\fn),V) \lra H^{i}(\Gamma_0(\fn),V).
\end{align*}
In addition, for every infinite place $v$ of $F$ splitting in $B$ there is an involution $T_v$ of the same space. All these operators commute. We denote by $\mathbb{T}$ the Hecke algebra, by which we mean the free polynomial ring over $\Z$ generated $\{T_\fl\}_{\fl\nmid\fd}$ and $\{T_v\}$, regarded as formal variables. The Hecke algebra $\mathbb{T}$ acts on the cohomology groups $H^{i}(\Gamma_0(\fn),V)$ by letting each formal variable act as the corresponding Hecke operator.

We will say that a cohomology class $f\in H^{n+s}(\Gamma_0(\fp\fm),\C)$ is an eigenclass if for all primes $\fl\nmid \fd$ we have that 
\begin{align*}
T_\fl f = a_\fl(f) f,\ \  \text{ for some } a_\fl(f)\in \C.  
\end{align*}
 In this case, the field $K_f=\Q(\{a_\fl(f)\}_{\fl\nmid \fd})$ is a number field and, since we are considering subgroups of the form $\Gamma_0$, it is in fact totally real.

We will say that $f$ is \emph{trivial} if $a_\fl(f)$ equals $ |\fl| + 1$ for all $\fl\nmid \fp\fm\fd$ (here $|\fl|$ stands for the norm of $\fl$). Finally, if there does not exist any $g\in H^{n+s}(\Gamma_0(\fm ' ),\C)$ with $\fm'\mid \fp\fm$ and $T_\fl g = a_\fl(f)g$ for all but finitely many primes $\fl$, then we say that $f$ is new. The following conjecture attaches an abelian variety to any nontrivial new eigenclass; it was formulated essentially in this form in \cite{Taylor}, the only difference being that below we make explicit the expected relation between the level of the newform and the conductor of the abelian variety.


\begin{conjecture}\label{conjecture: taylor}
  Let $f\in H^{n+s}(\Gamma_0(\fp\fm),\C)$ be a nontrivial new eigenclass of level $\fN = \fp\fm\fd$. Then one of the following holds:
  \begin{enumerate}
  \item There exists a simple abelian variety $A_f/F$ of dimension $d=[K_f:\Q]$ and conductor $\fN^d$ such that  $\Q\otimes\End(A)$ contains $K_f$ and 
  \begin{align*}
    \# A_f(\cO_F/\fl)= N_{K_f/\Q}(1+|\fl| - a_\fl(f)) \text{ for all } \fl \nmid \fN;
  \end{align*}
\item There exists a simple abelian variety $A_f/F$ of dimension $2d$ and conductor $\fN^{2d}$, and a quaternion division algebra $D$ over $K_f$ such that $\Q\otimes\End(A_f)$ contains $D$ and 
  \begin{align*}
    \# A_f(\cO_F/\fl)= N_{K_f/\Q}(1+|\fl| - a_\fl(f))^2 \text{ for all } \fl \nmid \fN.
  \end{align*}
  \end{enumerate}
Moreover, $A_f$ does not have complex multiplication (CM) defined over $F$, and if $F$ has at least one real embedding then $(2)$ does not occur.
\end{conjecture}
We note that in \cite{Taylor} the conjecture is stated for modular forms of arbitrary level, not necessarily of the form $\fN=\fp\fm\fd$ with $\fp\mid\mid \fN$. Our running assumption that there exists a prime dividing the level exactly is necessary for our construction, but not for the (conjectural) existence of $A_f$. However, this condition is also useful in addressing an issue concerning the case in which $A_f$ belongs to the second case in Conjecture \ref{conjecture: taylor}. The point is that the construction of Section \ref{section 3} associates to the eigenclass $f$ a $\fp$-adic lattice in $(\C_p^\times)^d$, which is conjectured to be the lattice of $A_f$. But a lattice in $(\C_p^\times)^d$ can only correspond to an abelian variety of dimension $d$ (see \S \ref{subsection: uniformization} below). Therefore, if for a given $f$ the variety $A_f$ turned out to be of dimension $2d$ it would not be clear a priory what our construction would be giving, even conjecturally. This is the question that we shall address now.

Suppose that $A_f$ is as in case $(2)$ of Conjecture \ref{conjecture: taylor}. In particular, $\dim A_f=2d$ and there is an injection of the quaternion division algebra $D$ into $\End(A_f)\otimes \Q$. Observe that, in fact, this must be an isomorphism $D\simeq \End(A_f)\otimes \Q$. Indeed, if we let $D'=\End(A_f)$ then $D'$ is a division algebra, since $A_f$ is simple. Then $D'$ acts on $H_1(A_{f,\C},\Q)$ which has dimension $4d$ over $\Q$; this implies that $[D'\colon \Q]\leq 4d$ and since $[D\colon \Q]=4d$ we have $D=D'$. Therefore, since $D$ is the endomorphism algebra of a simple abelian variety it belongs to either type II or type III in Albert's classification (see~\cite[\S 21, Theorem 2]{mumford}). In particular, $D$ is either totally definite or totally indefinite. 

The next proposition will show that the case where $D$ is totally indefinite cannot occur, since we are assuming that the conductor of $A_f$ has valuation $2d$ at $\fp$.

\begin{proposition}\label{prop: indefinite}
Let $A/F$ be a simple abelian variety of dimension $2d$ and conductor $\fL$. Let $K$ be a totally real number field with $[K\colon \Q]=d$ and $D$ a totally indefinite quaternion division algebra over $K$ with $D\simeq \Q\otimes \End(A)$.  Then $v_\fp(\fL)\geq 4d$ for each prime $\fp$ dividing $\fL$.
\end{proposition}
\begin{proof}
Let $\ell$ be a prime different from the residue characteristic of $\fp$. Let $V_\ell=T_\ell\otimes\Q_\ell$ be the rational $\ell$-adic Tate module of $A$, which is of rank $4d$ over $\Q_\ell$. Let $M$ be a maximal subfield of $D$; it is a quadratic extension of $K$ and therefore $[M:\Q]=2d$. Since $M$ is contained in $\End(A)\otimes \Q$, we have that $V_\ell$ is an $M\otimes\Q_\ell$ module of rank $2$ and $V_\ell\otimes_{\Q_\ell}\bar\Q_\ell$ breaks as the direct sum of $2d$ representations of dimension $2$, each of them conjugate to a given representation, say $\rho:\Gal(\bar F/F)\ra \Aut(V_\ell)$. The exponent at $\fp$ of the conductor of $\rho$ is equal to $\codim (V_\ell^{I_\fp})+ \delta_\fp$, where $I_\fp$ is the inertia at $\fp$ and $\delta_\fp$ is the Swan exponent. Therefore 
\begin{align*}
  v_\fp(\fL)=2d(\codim (V_\ell^{I_\fp})+\delta_\fp)
\end{align*}
and we see that $v_\fp(\fL)$ is a multiple of $2d$. We are assuming that $\fp$ divides $\fL$ and therefore that $A$ has bad reduction at $\fp$. By the criterion of N\'eron--Ogg--Shaffarevic this implies that $\codim(V_\ell^{I_\fp})\geq 1$ and therefore $v_\fp(\fL)\geq 2d$. In order to finish the proof, it is then enough to rule out the possibility that $v_\fp(\fL)=2d$; that is to say, we need to rule out the possibility that $\codim(V_\ell^{I_\fp})=1$ and $\delta_\fp=0$. 

Aiming for contradiction, assume that $\codim(V_\ell^{I_\fp})=1$ and $\delta_\fp=0$. Let $A'$ be the connected component of the special fiber of the N\'eron model of $A$ over $\cO_{F_\fp}$. It sits in an exact sequence
  \begin{align}\label{eq: chavelley}
    0 \lra T\times U\lra A' \lra B \lra 0,
  \end{align}
where $T$ is a torus, $U$ is a unipotent group, and $B$ is an abelian variety over the finite field $\cO_F/\fp$. If we let $t=\dim T$ and $u=\dim U$, then $t+u+\dim B = 2d$, implying $t \leq 2d$.

We claim that $A$ has potentially good reduction; that is to say, $t=0$ (cf. \cite[Theorem 3]{ribet-endomorphisms}). Indeed, if $t>0$ then by functoriality $D$ acts on $T$, and we get an inclusion 
\begin{align}\label{eq: inclusion of D}
D\hookrightarrow \End(T)\otimes \Q\simeq \M_t(\Q).
\end{align}
 We can interpret $\M_t(\Q)$ as $ \End_\Q(V)$, where $V$ is a $\Q$-vector space of dimension $t$. The inclusion \eqref{eq: inclusion of D} endows $V$ with the structure of a $D$-module. Now we see that 
 \begin{align*}
t=\dim_\Q (V) = \dim_D (V )\cdot \dim_\Q (D )= \dim_D (V)\cdot 4d 
 \end{align*}
and therefore $t\geq 4d$. But $t\leq 2d$ and this forces $t=0$, as we claimed.  

Therefore \eqref{eq: chavelley} can be written as
  \begin{align*}
    0 \lra U\lra A' \lra B \lra 0,
  \end{align*}
with $u+\dim B = 2d$. Since $\delta_\fp=0$ and $A$ has potentially good reduction, we have that $v_\fp(\fL)=2u$ (cf. \cite{ST} Theorem $4$ and the Remarks in page 500). Since we are assuming that $v_\fp(\fL)=2d$ this implies that $\dim B =d$.

Let $M$ be a maximal subfield of $D$, which we can choose to be totally real because $D$ is totally indefinite. By functoriality  we have an inclusion $D\hookrightarrow \End(B)\otimes \Q$. Let $B\sim B_0^{n_0}\times \cdots\times B_r^{n_r}$ be the decomposition of $B$ into simple varieties up to isogeny. Then $D'=\End(B_0^{n_0})\otimes \Q$ is a simple algebra with an embedding $D\hookrightarrow D'$. Let $L$ be the center of $D'$. By the results of Tate on endomorphism algebras of abelian varieties over finite fields \cite{Tate} we have that 
\begin{align*}
  2\dim(B_0^{n_0})=[L\colon\Q]\sqrt{[D'\colon L]}.
\end{align*}
See, for example, \cite[Theorem 8]{WM} where this is stated for simple varieties, from which the case of isotypic varieties follows directly. Therefore, since $\dim B_0^{n_0}\leq \dim B = d$ we see that  $[L\colon\Q]\sqrt{[D'\colon L]}\leq 2d$. But $[L\colon\Q]\sqrt{[D'\colon L]}$ is the dimension over $\Q$ of the maximal subfields of $D'$. Since we have an inclusion $D\hookrightarrow D'$, the field $M$ is also a subfield of $D'$, hence $M$ is a maximal subfield\footnote{In the context of $L$-simple algebras, by a subfield of $D'$ one understands a field contained in $D'$ and that contains $L$. But this is the case for $M$. Indeed, the compositum $LM$ is a subfield in this sense, and therefore its dimension over $\Q$ is $\leq 2d$; then it has to equal $2d$ and $LM=M$.} of $D'$ because $[M:\Q]=2d$. In particular $L$ is contained in $M$. Since $M$ is totally real, this implies that $L$ is totally real as well and that $D'$ is split at the real places of $L$ (because a maximal subfield of a simple algebra is a splitting field). But this contradicts \cite[Theorem 2 (d)]{Tate}, which asserts that the endomorphism algebra of an isotypic variety over a finite field does not split at any real place of the center.
\end{proof}

The same argument does not allow us to rule out the case where $D$ is totally definite. However, in this case $f$ is necessarily a modular form with complex multiplication. More precisely, the $L$-function of $A_f$ is a product of $L$-functions of Hecke characters of $F$.
\begin{proposition}\label{prop: cm}
Let $A/F$ be an abelian variety of dimension $2d$. Suppose that $D=\Q\otimes \End(A)$ is totally definite quaternion division algebra over a totally real number field $K$ and $[K\colon \Q]=d$.  Then there exist Hecke characters $\chi_i\colon \A_F^\times\ra \C^\times$ such that $L(A,s)=\prod_i L(s,\chi_i)$.  
\end{proposition}
\begin{proof}
Since $D$ is totally definite and $2[K\colon\Q]=\dim A$, a theorem of Shimura \cite[Proposition 15]{shimura-analytic-families} ensures that $A$ is $\Qbar$-isogenous to the square of a CM abelian variety $B$ of dimension $d$. Now the idea is to use the well known relation between $L$-functions of CM abelian varieties and Hecke characters. But we have to be careful in this case, because the complex multiplication of $A$ is not defined over $K$. Next, we will show that $A$ satisfies the hypothesis of \cite[\S 3]{milne}, and then Theorem 4 of loc. cit. implies the conclusion of the proposition.

To begin with, we remark that $B$ might not be simple. But it easy to see that $A_\Qbar$ is isotypical. Indeed, if $A_\Qbar\sim B_1\times B_2$ where $B_1$ and $B_2$ are not isogenous, then $D$ would act on each $B_i$; in particular it would act on $H_1(B_i,\Q)$ which have dimension $<4d$, and since $[D\colon \Q]=4d$ this is not possible. Therefore, we see that necessarily $A_{\Qbar} \sim C^r$ for some simple CM variety $C$. 

If we let $N=\End(C)$ then $\End(A_\Qbar)\simeq \M_r(N)$ and we can identify $N$ with $Z(\End(A_\Qbar))$. Let $L$ be the smallest extension of $F$ such that $Z(\End(A_\Qbar))\subset \End(A_L)$. We choose also a maximal subfield $M\subset D$ with the property that $M$ does not contain $N$.  This is possible because the intersection of two maximal subfields of $D$ is $K$, and $K$ does not contain $N$ because $N$ is a CM field and $K$ is totally real. We  identify $M$ with a subfield of $\Q\otimes\End(A_\Qbar)$ by means of the embedding $D\hookrightarrow \Q\otimes\End(A_\Qbar)$, and we let $E=M N$. It is a subfield of $\End(A_\Qbar)$, because $N$ commutes with $M$. Moreover, we have that $[E\colon \Q]=4d$; indeed, on the one hand $2d\mid [E:\Q]$ and since $N\not \subset M$ necessarily $[E:\Q]>2d$, but on the other hand $[E:\Q]\leq 2\dim A=4d$. 

Therefore, we have constructed a field $E$ such that $A_L$ has complex multiplication by $E$ defined over $L$. Moreover, $E$ is stable under the action of $\Gal(L/F)$. This is because $N$ is the center of $\End(A_\Qbar)$ and is stable, and $M$ consists on endomorphisms defined over $K$ and it is also stable. Therefore, we are in the assumptions of \cite[Theorem 4]{milne} (they are stated in the last paragraph of p. 186), and the proposition follows from this theorem. 
\end{proof}

We assume for the rest of the article  that $f$ does not have CM. In view of propositions \ref{prop: cm} and \ref{prop: indefinite}, this implies that $A_f$ has dimension $d$, rather than $2d$.

The next step is to show that $A_f$ has purely multiplicative reduction at $\fp$ (this implies that $A_f$ admits a $\fp$-adic uniformization, see \S \ref{subsection: uniformization} below). As before let $A'$ be the connected component of the special fiber of the N\'eron model of $A_f$ over $\cO_{F_\fp}$, which sits in an exact sequence
  \begin{align}
    0 \lra T\times U\lra A' \lra B \lra 0,
  \end{align}
where $T$ a torus, say of dimension $t$, $U$ is a unipotent group, and $B$ is an abelian variety over the finite field $\cO_F/\fp$. The variety $A_f$ is said to have purely multiplicative reduction at $\fp$  if $t=d$. 

Recall our running assumption that $\fp||\fN$, which implies that the exponent of $\fp$ in the conductor of $A_f$ is $d$.  It is a well known result that if an elliptic curve has exponent at $\fp$ of the conductor equal to $1$, it has multiplicative reduction at $\fp$. The following is a generalization of this statement to dimension $d>1$, under the assumption that there is a totally real number field of degree $d$ acting on the abelian variety.

\begin{proposition}
  Let $A/F$ be an abelian variety of dimension $d$, equipped with an embedding $K\hookrightarrow \End(A)\otimes\Q$ of a totally real number field $K$ of degree $d$ over $\Q$. Suppose that $\fp$ is a prime such that the exponent of the conductor of $A$ at $\fp$ is $d$. Then $A$ has purely multiplicative reduction at $\fp$.
\end{proposition}

\begin{proof}
First of all we claim that if $A$ does not have potentially good reduction at $\fp$, then it has purely multiplicative reduction. Indeed, if $A$ does not have potentially good reduction then $t>0$. By functoriality $K$ acts on the torus $T$, and we get an inclusion $$K\hookrightarrow \End(T)\otimes \Q\simeq M_t(\Q).$$ This implies that $d\leq t$, and since we already know that $t\leq d$ this proves the claim.

Therefore, in order to finish the proof it is enough to show that  $A$ does not have potentially good reduction. Let $\ell$ be a prime different from $\fp\cap\Z$ and consider the Tate module $V_\ell(A)=T_\ell(A)\otimes\Q$. Since $[K\colon\Q]=\dim A$ we know that $V_\ell(A)\otimes\bar\Q_\ell$ is the direct sum of $d$ representations of dimension $2$, all conjugate to a given one, say to $\rho\colon \Gal(\bar F/F)\to V_\ell$. Since they are all conjugate, each of them has conductor exponent at $\fp$ exactly $1$. The exponent of the conductor of $\rho$ at $\fp$ is $\codim (V_\ell^{I_\fp})+\delta_\fp$, where $I_\fp$ is the inertia subgroup of any extension of $\fp$ to $\bar K$ and $\delta_\fp$ is the Swan part of the exponent.  By the criterion of N\'eron--Ogg--Shafarevich, since $A$ has bad reduction at $\fp$ necessarily $\codim (V_\ell^{I_\fp})\geq 1$.  This implies that $\delta_\fp=0$ and  $\dim V_\ell^{I_\fp} = 1$. Thus $\rho_{|I_\fp}$ is of the form $\smtx{\psi}{\star}{0}{1}$ for some character $\psi$. Since $K$ is totally real,  by \cite[Lemma 4.5.1]{ribet-galois-action} the determinant of $V_\ell$ is the (unramified) cyclotomic character. Therefore $\psi=1$ and $\rho_{|I_\fp}=\smtx{1}{\star}{0}{1}$. In particular, the image of $I_\fp$ under $\rho$ is infinite. By \cite[Theorem 2]{ST} this implies that $A$ does not have potentially good reduction at $\fp$.
\end{proof}

\subsection{Uniformization and $p$-adic lattices}\label{subsection: uniformization} We next recall the basic facts that we will use on $p$-adic uniformization of abelian varieties. Let $A/F$ be an abelian variety of dimension $d$ with purely multiplicative reduction at $\fp$.  Then $A$ admits a $\fp$-adic uniformization: there exist free abelian groups $X,Y$, and a pairing
\begin{align}\label{eq: pairing}
i\colon  X \times Y \lra F_\fp^\times
\end{align}
such that the composition $\ord_\fp \circ i\colon X\otimes \Q \times Y\otimes \Q \ra \Q$ is a perfect pairing and induces an isomorphism
\begin{align*}
 A(\bar{F}_\fp)\simeq  \Hom(Y,\bar F_\fp^\times) / \Lambda,
\end{align*}
where $\Lambda$ is the image of $X$ in $\Hom(Y,\bar F_\fp^\times)$ under the map induced by $i$. The subgroup $$\Lambda\subset  \Hom(Y, F_\fp^\times)\simeq  (F_\fp^\times)^d$$ is a lattice, i.e., a free discrete subgroup of rank $d$.

We will be interested in how the lattices of isogenous abelian varieties are related. Suppose that $A$ and $A'$ are abelian varieties over $F_\fp$, uniformized by lattices $\Lambda$ and $\Lambda'$ in $ (F_\fp^\times)^d$. Let $\{v_1,\dots,v_d\}$ and $\{w_1,\dots,w_d\}$ be bases of $\Lambda$ and $\Lambda'$ respectively. Put $V=(v_{ij})\in M_d(F_\fp)$ and $W=(w_{ij})\in M_d(F_\fp)$ the matrices whose columns are the vectors of these bases. Following a notation introduced in \cite{kadziela}, for $B=(b_{ij})\in\M_d(\Z)$ we denote by $V^B$ the matrix with entries
\begin{align*}
 b_{ij}= v_{1j}^{b_{i1}}v_{2j}^{b_{i2}}\cdots v_{dj}^{b_{id}}.
\end{align*} 
Similarly, for $C=(c_{ij})\in\M_d(\Z)$ we denote by ${}^C W$ the matrix with entries
\begin{align*}
c_{ij}=  w_{i1}^{c_{1j}}w_{i2}^{c_{2j}}\cdots w_{id}^{c_{dj}}.
\end{align*}
Alternatively, these matrices can also characterized as follows. Let $\lambda\colon F_\fp^\times\to R$ be any group homomorphism of $F_\fp^\times$ to the additive group of a ring $R$, and for a matrix $U=(u_{ij})$ denote by $\lambda(U)$ the matrix with entries $\lambda(u_{ij})$; then
\begin{align}\label{eq: logs}
  \lambda(V^B)=B\lambda(V)\ \ \text{and} \ \ \lambda({}^C W)= \lambda(W) C.\quad\text{ for all $\lambda$.}
\end{align}
The following result of Kadziela characterizes isogenies of abelian varieties over $F_\fp$ in terms of their uniformizing lattices.
\begin{theorem}[\cite{kadziela}, Theorem $3$]\label{th: kadziela}
The abelian varieties $A$ and $A'$ are isogenous if and only if there exist matrices $B,C\in \M_d(\Z)$ such that $V^B={}^C W$.   
\end{theorem}

\section{Integration pairing and construction of the lattice}\label{section 3}

Let $f\in H^{n+s}(\Gamma_0(\fp\fm),\C)$ be a new eigenclass and let $K_f=\Q(\{a_\fl(f)\})$ be the number field, say of degree $d$, generated by the eigenvalues of $f$. We suppose that $f$ is also an eigenclass with eigenvalue $+1$ for all the involutions at infinity\footnote{We remark that the construction of the lattice works for any choice of signs at infinity, and we expect Conjecture \ref{conj: main} to hold for any choice. However, our numerical experiments have been done for eigenclasses with eigenvalues $+1$ at infinity, so we prefer to consider this case.}; that is, $T_v f=f$ for all real places $v$ of $F$ that split in $B$. Then $f$ gives rise to a character of the Hecke algebra 
\begin{align*}
\lambda = \lambda_f\colon \mathbb{T}\lra \C,
\end{align*}
 via the formulas $\lambda(T_\fl)=a_\fl(f)$ for all $\fl\nmid  \fd$ and $\lambda(T_v)=1$ for the infinite places $v$. For any embedding $\sigma\colon K_f\hookrightarrow \C$ there is a conjugate newform $\acc\sigma f$, characterized by the fact that its character, that we will denote $\lambda_\sigma$, is given by $\lambda_\sigma(T_\fl)=\sigma(\lambda(T_\fl))$ and $\lambda_\sigma (T_v)=1$. 

We denote by 
\begin{align}\label{eq: f-component}
H^{n+s}(\Gamma_0(\fp\fm),\Q)^{f}\subset H^{n+s}(\Gamma_0(\fp\fm),\Q)
\end{align}
 the $\mathbb{T}$-irreducible subspace such that $f$ belongs to $ H^{n+s}(\Gamma_0(\fp\fm),\Q)^{f}\otimes\C$. Since $f$ is a newform,   by  \emph{multiplicity one} this space decomposes over $\C$ as the sum of $d$ one-dimensional $\mathbb{T}$-eigenspaces:
\begin{align*}
  H^{n+s}(\Gamma_0(\fp\fm),\Q)^{f}\otimes\C =\bigoplus_{\sigma:K_f\hookrightarrow \C}H^{n+s}(\Gamma_0(\fp\fm),\C)^{\lambda_\sigma},
\end{align*}
where for any $\C \otimes \mathbb{T}$-module $M$ and any character $\alpha:\mathbb{T}\ra \C$ we put
\begin{align*}
  M^\alpha = \{m\in M \colon Tm = \alpha(T)m\text{ for all } T\in\mathbb{T}\}.
\end{align*}
Define also
\begin{align*}
H^{n+s}(\Gamma_0(\fp\fm),\Z)^f=  H^{n+s}(\Gamma_0(\fp\fm),\Q)^{f} \cap  H^{n+s}(\Gamma_0(\fp\fm),\Z).
\end{align*}
The Hecke algebra also acts on the homology groups, with the same systems of Hecke eigenvalues. Similarly as before, we define $H_{n+s}(\Gamma_0(\fp\fm),\Q)^f$ to be the Hecke constituent such that the system of Hecke eigenvalues $\{a_\fl(f)\}_{\fl\nmid \fd}$ arises in $H_{n+s}(\Gamma_0(\fp\fm),\Q)^f\otimes \C$, and
\begin{align*}
H_{n+s}(\Gamma_0(\fp\fm),\Z)^f=  H_{n+s}(\Gamma_0(\fp\fm),\Q)^{f} \cap  H_{n+s}(\Gamma_0(\fp\fm),\Z).
\end{align*}
Note that both $H^{n+s}(\Gamma_0(\fp\fm),\Z)^f$ and $H_{n+s}(\Gamma_0(\fp\fm),\Z)^f$ are free abelian groups of rank $d$.

 In this section we will recall and slightly generalize the constructions of \cite{GMS2}, which can be formulated as the existence of a multiplicative integration pairing 
\begin{align}\label{eq: pairing between cohomology}
  \langle \cdot , \cdot \rangle : H^{n+s}(\Gamma_0(\fp\fm),\Z)^f\times H_{n+s}(\Gamma_0(\fp\fm),\Z)^f \lra F_{\fp^2}^\times, 
\end{align}
where $F_{\fp^2}$ stands for the quadratic unramified extension of $F_\fp$. We will then conjecture that this pairing can be identified with the uniformization pairing \eqref{eq: pairing} for $A_f$. The construction, which follows very closely the ideas introduced in \cite{Gr}, is based on passing to the cohomology of a certain $S$-arithmetic group $\Gamma$ related to $\Gamma_0(\fm)$ and $\Gamma_0(\fp\fm)$.

\label{def-of-gamma}Let $\cO_{F,\{\fp\}}$ denote the elements of $F$ with non-negative valuation at the primes different from 
$\fp$, and set  $$R=R_0(\fm)\otimes_{\cO_F}\cO_{F,\{\fp\}}.$$ As usual, $R_1^\times$ stands for the group of norm $1$ elements of $R$, and we denote by $\Gamma$ the image of $R_1^\times$ in $B^\times /F^\times$. 

Recall that $B$ splits at $\fp$. By fixing an isomorphism $B\otimes_F F_\fp \simeq \M_2(F_\fp)$ we can regard $\Gamma$ as a subgroup of $\PGL_2(F_\fp)$. Similarly as before, for any $\PGL_2(F_\fp)$-module $V$ the (co)homology groups $H^i(\Gamma,V)$ and $H_i(\Gamma,V)$ are equipped with the action of Hecke operators $T_\fl$ for $\fl\nmid \fp\fd$ and involutions at infinity $T_v$ for the infinite places of $F$ that split in $B$. 

Let $A$ be an abelian group. An $A$-valued measure on $\P^1(F_\fp)$ is a function
\begin{align*}
  \omega\colon \{ \text{Open compact subgroups of } \P^1(F_\fp)\}\lra A
\end{align*}
such that $\omega (U_1\cup U_2)=\omega(U_1) + \omega(U_2)$ if $U_1$ and $U_2$ are disjoint. We denote by $\Meas_0(\P^1(F_\fp),A)$ the set of such measures which in addition satisfy that $\omega(\P^1(F_\fp))=0$.  There is a natural action of $\PGL_2(F_\fp)$ on measures, and therefore also of $B^\times/F^\times$, by means of $(g\omega)(U)=\omega(g^{-1}U)$. In particular, there are Hecke operators acting on the cohomology groups $H^i(\Gamma,\Meas_0(\P^1(F_\fp),\Z))$. 

Let $\cH_\fp=F_{\fp^2}\setminus F_\fp$ denote the $F_{\fp^2}$-rational points of the $p$-adic upper half plane. Given $\omega \in\Meas_0(\P^1(F_\fp),\Z)$ and $x,y\in \cH_\fp$ the multiplicative integral $\Xint\times_{x}^{y}\omega$ is defined as
\begin{align*}
  \Xint\times_{x}^{y}\omega = \lim_\cU \prod_{U\in\cU} \left(\frac{t_U-y}{t_U-x}\right)^{\omega(U)}\in F_{\fp^2}^\times,
\end{align*}
where $\mathcal{U}$ runs over the coverings of $\P^1(F_\fp)$ by open-compacts with diameter tending to zero, and $t_U$ is any sample point in $U$. This can be seen as a pairing
\begin{align*}
\Xint\times\colon  \Meas_0(\P^1(F_\fp),\Z) \times \Div^0 \cH_\fp \lra F_{\fp^2}^\times,
\end{align*}
which induces a corresponding pairing in (co)homology, via cap product:
\begin{align*}
  \Xint\times \colon H^{n+s}(\Gamma,\Meas_0(\P^1(F_\fp),\Z)) \times H_{n+s}(\Gamma,\Div^0 \cH_\fp) \lra F_{\fp^2}^\times.
\end{align*}
Finally, we proceed to explain how this gives rise to the pairing \eqref{eq: pairing between cohomology}.

The first step is to show that there is a subspace of $H^{n+s}(\Gamma,\Meas_0(\P^1(F_\fp),\Q))$  which is isomorphic to $H^{n+s}(\Gamma_0(\fp\fm),\Q)^f$ as a Hecke module.  We will make use of the following notation: let $$I_f=\mathrm{Ann}_{\mathbb{T}}(H^{n+s}(\Gamma_0(\fp\fm),\Q)^f)$$ denote the annihilator of $\mathbb{T}$ acting on the irreducible space $H^{n+s}(\Gamma_0(\fp\fm),\Q)^f$; then, for any $\mathbb{T}$-module $M$ we define
\begin{align*}
  M^f = \bigcap_{T\in I_f}\ker(T)\subset M.
\end{align*}
Observe that the notation is consistent with the one introduced before, in the sense that $H^{n+s}(\Gamma_0(\fp\fm),\Q)^f$, as introduced in \eqref{eq: f-component}, indeed coincides with $\cap_{T\in I_f}\ker(T)$. The following is a generalization of \cite[Proposition 25]{Gr} and \cite[Proposition 4.5]{GMS}.
\begin{proposition}\label{prop: alpha}
  There is a natural Hecke equivariant map
  \begin{align*}
\rho:    H^{n+s}(\Gamma,\Meas_0(\P^1(F_\fp),\Q))\lra H^{n+s}(\Gamma_0(\fp\fm),\Q),
  \end{align*}
which induces an isomorphism
\begin{align*}
 \rho\colon   H^{n+s}(\Gamma,\Meas_0(\P^1(F_\fp),\Q))^f\simeq  H^{n+s}(\Gamma_0(\fp\fm),\Q)^f.
\end{align*}
\end{proposition}
\begin{proof}
  By \cite[Display (4.6)]{GMS} there is a Hecke equivariant homomorphism 
\begin{align*}
\rho\colon   H^{n+s}(\Gamma,\mathrm{HC}(\Q))\lra H^{n+s}(\Gamma_0(\fp\fm),\Q)^{\fp-\text{new}}.
\end{align*}
Here $\mathrm{HC}(\Q)$ stands for the module of $\Q$-valued harmonic cocycles on the Bruhat--Tits tree of $\PGL_2(F_\fp)$, which is isomorphic to $\Meas_0(\P^1(F_\fp),\Q)$. Moreover, there is a short exact sequence 
\begin{align*}
  0 \lra \coker \alpha \lra  H^{n+s}(\Gamma,\mathrm{HC}(\Q))\stackrel{\rho}{\lra} H^{n+s}(\Gamma_0(\fp\fm),\Q)^{\fp-\text{new}} \lra 0,
\end{align*}
where $\alpha$ is the map denoted as $\alpha_{n+s+1}$ in \cite{GMS}.  

To finish the proof, we follow the argument of \cite[Proposition 25]{Gr}. Indeed, for any $\mathbb{T}$-module $M$ we have that $M^f=\Hom_{\mathbb{T}}(\mathbb{T}/I_f,M)$. Applying the functor $\Hom_{\mathbb{T}}(\mathbb{T}/I_f,\cdot)$ to the above exact sequence we obtain
\begin{align*}
  0 \lra (\coker \alpha)^f \lra  H^{n+s}(\Gamma,\mathrm{HC}(\Q))^f\lra (H^{n+s}(\Gamma_0(\fp\fm),\Q)^{\fp-\text{new}} )^f \lra \operatorname{Ext}^1_{\mathbb{T}}(\mathbb{T}/I_f,\coker \alpha).
\end{align*}
To finish the proof, it remains to show that both $(\coker \alpha)^f$ and $\operatorname{Ext}^1_{\mathbb{T}}(\mathbb{T}/I_f,\coker \alpha)$ are zero.

Since cuspidal cohomology vanishes in degree $<n+s$, the Hecke operators $T_\fl$ act on $H^{n+s-1}(\Gamma_0(\fp\fm),\Q)$ as multiplication by $|\fl| + 1$. By construction, $\coker \alpha$ is a quotient of $H^{n+s-1}(\Gamma_0(\fp\fm),\Q)^2$; this implies that $T_\fl$ acts on $\coker \alpha$ as multiplication by $|\fl| + 1$. Since $f$ is cuspidal, we see that $(\coker\alpha)^f=0$. 

Finally, we show that $\operatorname{Ext}^1_{\mathbb{T}}(\mathbb{T}/I_f,\coker \alpha)^f = 0$ as well. Suppose that 
\begin{align}\label{eq: extension}
0\lra \coker \alpha \lra E\stackrel{\pi}{\lra}\mathbb{T}/I_f\lra 0
\end{align}
is a $\mathbb{T}$-module extension of $\mathbb{T}/I_f$ by $\coker\alpha$, and we will see that it splits. Let $s'$ be a section of $\pi$ as $\Q$-vector spaces. This might not be a $\mathbb{T}$-module homomorphism, but it can be modified to obtain a section of $\mathbb{T}$-modules:  observe that $T_\fl-|\fl|-1$ acts on $\mathbb{T}/I_f\simeq K_f$ as multiplication by $d_\fl:= a_\fl(f) - |\fl|-1$, and since $f$ is cuspidal   $d_\fl\neq 0$; then we define $s\colon \mathbb{T}/I_f\ra E$ as 
\begin{align*}
  s(x) = (T_\fl - |\fl|-1)s'(\frac{x}{d_\fl}).
\end{align*}
It is easy to check that $s$ is indeed a section of $\mathbb{T}$-modules, and therefore \eqref{eq: extension} splits.
\end{proof}

The group $\Gamma$ is isomorphic to the amalgamated product of two copies of $\Gamma_0(\fm)$ over $\Gamma_0(\fp\fm)$ (see, e.g., \cite{serre-trees}). The Mayer--Vietoris sequence in this setting \cite[Chapter II, \S 7]{brown} is then
\begin{align*}
 \cdots\lra  H_{n+s+1}(\Gamma_0(\fm),\Q)^2\lra H_{n+s+1}(\Gamma,\Q)\stackrel{\eta}{\lra}H_{n+s}(\Gamma_0(\fp\fm),\Q) \stackrel{\partial}{\lra} H_{n+s}(\Gamma_0(\fm),\Q)^2\lra\cdots
\end{align*}
\begin{lemma}\label{lemma: d}
  The map $\eta$ induces an isomorphism
\begin{align*}
\eta\colon H_{n+s+1}(\Gamma,\Q)^f\stackrel{\simeq}{\lra}H_{n+s}(\Gamma_0(\fp\fm),\Q)^f.
\end{align*}
\end{lemma}
\begin{proof}
We claim that $H_{n+s}(\Gamma_0(\fp\fm),\Q)^f$ lies in the kernel of $\partial$. Indeed, it is enough to show this for $H_{n+s}(\Gamma_0(\fp\fm))^f\otimes \C$; but this has a basis of eigenclasses which are new, and the fact that they are new at $\fp$ is equivalent to being in the kernel of $\partial$. Moreover, $\eta$ is injective because $H_{n+s+1}(\Gamma_0(\fm),\Q)^f = 0$ (since $f$ is a newform of level $\fp\fm$, its system of eigenvalues can not be found in lower level).
\end{proof}

Now consider the exact sequence of $B^\times/F^\times$ modules
\begin{align*}
  0 \lra \Div^0 \cH_\fp\lra \Div \cH_\fp \lra \Z \lra 0.
\end{align*}
It induces a long exact sequence in group homology; we are interested in the connecting homomorphism 
\begin{align*}
  \delta\colon H_{n+s+1}(\Gamma,\Z)\lra H_{n+s}(\Gamma,\Div^0 \cH_\fp).
\end{align*}
Finally, we define the pairing 
\begin{align}\label{eq: pairing 2}
  \langle \cdot , \cdot \rangle_f : H^{n+s}(\Gamma_0(\fp\fm),\Z)^f\times H_{n+s}(\Gamma_0(\fp\fm),\Z)^f \lra F_{\fp^2}^\times
\end{align}
as follows: given $\omega \in H^{n+s}(\Gamma_0(\fp\fm),\Z)^f$ and $\gamma \in  H_{n+s}(\Gamma_0(\fp\fm),\Z)^f$, then
\begin{align*}
  \langle \omega, \gamma \rangle_f = \Xint\times_{\delta \circ d^{-1}\gamma} \rho^{-1}(\omega). 
\end{align*}
We conjecture that the pairing \eqref{eq: pairing 2} uniformizes $A_f/F_\fp$ up to isogeny. To make this statement more explicit, let $\gamma_1,\dots.\gamma_d$ be a basis of $H_{n+s+1}(\Gamma_0(\fp\fm),\Z)^f$ and $\omega_1,\dots,\omega_d$ a basis of $H^{n+s+1}(\Gamma_0(\fp\fm),\Z)^f$. Define $\Lambda_f'\subset (F_{\fp^2}^\times)^d$ to be the subgroup generated by the $d$ vectors
\begin{align}\label{eq: basis of Lambda}
 \left( \langle \omega_i, \gamma_1, \rangle, \dots, \langle \omega_i, \gamma_d, \rangle  \right),\ \ \  i =1,\dots, d .
\end{align}
Also, denote by $\Lambda_f\subset (F_\fp^\times)^d$ the $p$-adic lattice of $A_f$.

\begin{conjecture}\label{conj: main}
  The vectors of \eqref{eq: basis of Lambda}  belong to $(F_\fp^\times)^d$, and $(F_\fp^\times)^d/\Lambda_f'$ is isogenous to $(F_\fp^\times)^d/\Lambda_f$.
\end{conjecture}
Some instances of this conjecture are known in the case $d=1$. If $F=\Q$ and $B=\M_2(\Q)$ this is a theorem of Darmon \cite{darmon-integration}. For $B$ a quaternion division algebra over $\Q$ it was proven independently in \cite{greenberg-dasgupta} and \cite{LRV}. For higher weight the result is due to Seveso \cite{seveso}, and for totally real $F$ some cases are proven in \cite[Proposition 5.9]{spiess}. In the next section we provide some numerical evidence for the conjecture in the case $d=2$ and $F$ a cubic field of signature $(1,1)$.

\begin{remark}
Using the same arguments as in \cite[\S 11]{Gr}, which appear also in more detail in \cite{rotger-seveso} and \cite{greenberg-seveso}, one can see that the pairing
\begin{align}\label{eq: pairing ord}
 \ord_\fp\circ \langle \cdot , \cdot \rangle_f : H^{n+s}(\Gamma_0(\fp\fm),\Z)^f\times H_{n+s}(\Gamma_0(\fp\fm),\Z)^f \lra \Q
\end{align}
is non-degenerate, which implies that $\Lambda_f'$ is a lattice. The non-degeneracy of \eqref{eq: pairing ord} is a consequence, on the one hand, of the naturality of the several (co)homological maps involved in the definition of the pairing and, crucially, of the combinatorial description of $ \ord_\fp\circ \langle \cdot , \cdot \rangle_f$ as stated, for instance, in \cite[Lemma 28]{Gr}. If we denote by $\eta^*\colon H^{n+s}(\Gamma_0(\fp\fn),\Q)^f \simeq H^{n+s+1}(\Gamma,\Q)^f$ the map arising from $\eta$ via the universal coefficients theorem, then the penultimate display in \cite[p. 573]{Gr} shows that $\ord_\fp \langle \eta^*(-),\eta(-)\rangle_f$ coincides with the natural pairing 
\begin{align*}
  H^{n+s+1}(\Gamma,\Z)^f\times H_{n+s+1}(\Gamma,\Z)^f \lra \Q.
\end{align*}

\end{remark}

\section{The case of abelian surfaces: calculations and numerical evidence}\label{section 4}
In this \S{} we present some computational evidence for Conjecture~\ref{conj: main} in the case where $A_f$ is of dimension $2$. We report on the numerical calculation of the lattice $\Lambda_f'$ for a concrete modular form $f$, which turns out to coincide (up to our working precision of $50$ $p$-adic digits) with a lattice which is isogenous to the lattice of $A_f$. 

In \S\ref{Algorithms for computing the $p$-adic lattice} we briefly describe the algorithms that we have used to compute the period lattice $\Lambda_f'$, which are actually an adaptation of the ones described in \cite{GMS2} for the case of elliptic curves. In \S\ref{subsection: p-adic periods of genus two curves} we review Teitelbaum's explicit formulas to compute the $p$-adic lattice of the Jacobian of a genus two curve of \cite{Tei}, which we use to compute the lattice of $A_f$. In \S\ref{subsection: a numerical verification} we present the results of our numerical calculations of $\Lambda_f'$ for a modular form $f$ over a number field $F$ of signature $(1,1)$. In this case,  $A_f$ turns out to be isogenous to the Jacobian of a genus two curve $C_f$ for which we know an explicit equation. Using Teitelbaum's formulas we can compute its period lattice $\Lambda_f$, and check that it is isogenous (up to high precision) to the lattice $\Lambda_f'$.

\subsection{Algorithms for the computation of the $p$-adic lattice}\label{Algorithms for computing the $p$-adic lattice} The code that we use to compute the pairing \eqref{eq: pairing 2} in dimension $2$ can be found at \texttt{https://github.com/mmasdeu/darmonpoints}. The algorithms are mainly an adaptation of the ones presented in \cite{GMS2} for the one dimensional case, and thus we give a brief overview of then,  emphasizing the points which require to be modified in higher dimension. A complete discussion of the details that we omit can be found in \cite[\S 3]{GMS2}.

Recall our running notation: $F$ is a number field of narrow class number one and $\fN\subset \cO_K$ is an ideal that decomposes into coprime ideals $\fN=\fp\fm\fd$, where $\fp$ is prime and $\fd$ is squarefree. Also $B/F$ is a quaternion algebra of discriminant $\fd$ and $\Gamma_0(\fp\fm)\subset\Gamma_0(\fm)$ are the norm $1$ units of Eichler orders of level $\fp\fm$ and $\fm$, respectively. We denote by $n$ the number of split real places of $B$ and by $s$ the number of complex places of $F$. Our goal is, first of all, to decide whether there exists a newform $f\in H^{n+s}(\Gamma_0(\fp\fm),\C)$ with $[K_f:\Q]=2$; in case it does, we aim to compute it and to compute the integration pairing \eqref{eq: pairing 2} and the lattice $\Lambda_f'$.  The algorithms that we next describe only work under the assumption that $n+s = 1$. In other words, if $B$ splits at a single infinite place. From now on we assume this:
\begin{ass}\label{ass}
  $n+s = 1$.
\end{ass}
There are two reasons for this restriction. The first one is that the (co)homology groups are then in degree $1$, and higher degrees are more difficult to deal with algorithmically; the second is that under Assumption \ref{ass} we can use the algorithms of Voight \cite{voight} (if $n=1$) and Page \cite{page} (if $s=1$) that compute generators and relations of $\Gamma_0(\fp\fm)$ and provide an effective solution of the word problem. 

If $G$ is a group and $A$ a $G$-module, a convenient  way of thinking of the homology groups is in terms of the bar resolution, in which the group of $i$-th chains are $$C_i(G,A)=\Z[G]\otimes_\Z \stackrel{(i)}{\cdots} \otimes_\Z \Z[G] \otimes_\Z A.$$
In order to compute $H_1(\Gamma_0(\fp\fm),\Z)$, it is also useful to use the canonical isomorphism with the abelianization
\begin{align*}
  H_1(\Gamma_0(\fp\fm),\Z)\simeq \Gamma_0(\fp\fm)^{\text{ab}}.
\end{align*}
From the generators and relations for $\Gamma_0(\fp\fm)$ it is straightforward to obtain generators and relations for the $\Gamma_0(\fp\fm)^\text{ab}$.

Then computing the Hecke operator at a prime $\fl\nmid \fN$ essentially boils down to finding an element $\pi_\fl\in R_0(\fp\fm)$ whose reduced norm generates $\fl$ and is positive at the real places. This can be done with the routines for quaternion algebras implemented by Voight in~\cite{voight2005quadratic} and available in Magma \cite{magma}. Once $\pi_\fl$ is found one can find a decomposition of the double coset $  \Gamma_0(\fp\fm)\pi_\fl \Gamma_0(\fp\fm)$ of the form
\begin{align*}
  \Gamma_0(\fp\fm)\pi_\fl \Gamma_0(\fp\fm)= \bigsqcup_{i=0}^{|\fl|}g_i\Gamma_0(\fp\fm).
\end{align*}
If $c=\sum_g n_g \cdot g\in \Z[G]$ is a cycle, the Hecke operator acting on the homology class $[c]$ is given by the explicit formula
\begin{align*}
  T_\fl([c])= \sum_{i=0}^{|\fl|}\sum_g n_g\cdot t_i(g),
\end{align*}
where $t_i(g)$ is defined by the equality $g^{-1}g_i = g_j t_i(g)^{-1}$ for some (unique) $j$. Similar explicit formulas exist for the operators $T_\fl$ for $\fl\mid \fp\fm$ and for the Hecke operators at the infinite places $T_v$. Now, if ${c_1,\dots,c_t}$ is basis of $H_1(\Gamma_0(\fp\fm),\Z)$ one can compute $T_\fl(c_i)$ using the above formula, and express it in terms of the basis by using the explicit solution to the word problem provided by the algorithms of Voight and Page. We are not interested in torsion homology classes, so if (say) $c_1,\dots,c_m$ are the free generators we obtain an $m\times m$ matrix of $T_\fl$ acting on the torsion free part of $H_1(\Gamma_0(\fp\fm),\Z)$; equivalently, we can also think that this is the matrix of $T_\fl$ acting on $H_1(\Gamma_0(\fp\fm),\Q)$.

We have implemented the algorithms under the assumption that the ideal $\fm$ is trivial. This is not a restriction of the method and it could be dispensed with, but we restricted to this case in order to simplify  some of the steps in the calculation. So let us suppose from now on that $\fm = (1)$.

The first step is to compute the $\fp$-new part of $H_1(\Gamma_0(\fp),\Q)$. Let $\omega_\fp\in R_0(\fp)^\times$ be an element whose reduced norm is positive at the real places of $F$ and that normalizes $\Gamma_0(\fp)$. Then, if we let $\widehat \Gamma_0(1)=\omega_\fp^{-1}\Gamma_0(1)\omega_\fp$ we can identify the group $\Gamma$ introduced in page~\pageref{def-of-gamma} with the amalgamated product $\Gamma_0(1)\star_{\Gamma_0(\fp)} \widehat \Gamma_0(1)$ (all the groups viewed as subgroups of $\Gamma$). The inclusions $ \Gamma_0(\fp)\subset \Gamma_0(1)$ and $\Gamma_0(\fp)\subset \widehat \Gamma_0(1)$ induce maps
\begin{align*}
  \alpha\colon H_1(\Gamma_0(\fp),\Q)\lra H_1(\Gamma_0(1),\Z),\ \   \hat\alpha\colon H_1(\Gamma_0(\fp),\Z)\lra H_1(\widehat\Gamma_0(1),\Z).
\end{align*}
Then the $\fp$-new part is $H_1(\Gamma_0(\fp),\Q)_{\fp-\text{new}}=\ker(\alpha)\cap \ker (\hat \alpha)$, which can be easily computed from the generators and relations of $\Gamma_0(\fp)^{\text{ab}}$ and $\Gamma_0(1)^{\text{ab}}$. We compute next $H_1(\Gamma_0(\fp),\Q)_{\fp-\text{new}}^+$, the subspace on which $T_v$ acts as $+1$ where $v$ is the real place of $F$ that splits in $B$ (this is only needed if $n=1$). Now we compute the matrix of $T_{\fl_1}$ acting on  $H_1(\Gamma_0(\fp),\Q)_{\fp-\text{new}}^+$ for some small prime $\fl_1\nmid \fN$. This module decomposes into a direct sum of submodules, given by the factorization of the minimal polynomial of $T_{\fl_1}$ into coprime factors. If some factor is irreducible, then the corresponding submodule is an irreducible $\mathbb{T}$-module. To each non-irreducible submodule, we apply the same procedure for the Hecke operator $T_{\fl_2}$ for some other prime $\fl_2$, and so on. In this way, after applying a finite number of Hecke operators $T_{\fl_1},T_{\fl_2},T_{\fl_3},\dots$ we will have decomposed $H_1(\Gamma_0(\fp),\Q)_{\fp-\text{new}}^+$ into a direct sum of irreducible $\mathbb{T}$-modules. Then the submodules of rank $2$ correspond to newforms $f$ with $[K_f:\Q]=2$, and the submodule is what we denoted $H_1(\Gamma_0(\fp),\Q)^f$. At this point we have explicitly computed a basis $\tilde\gamma_1,\tilde\gamma_2$ of $H_1(\Gamma_0(\fp),\Q)^f$. Since we are interested in integral classes, we set $\gamma_1=\tilde\gamma_1^a$ and $\gamma_2=\tilde \gamma_2^b$, for $a,b\in\Z$ such that $\gamma_1,\gamma_2\in H_1(\Gamma_0(\fp),\Z)^f$. 

The next step is for each $\gamma\in\{ \gamma_1,\gamma_2\}$ to compute $\delta\circ d^{-1}(\gamma)\in H_1(\Gamma,\Div^0 \cH_\fp)$. By construction $\gamma$ lies in $ \ker(\alpha)\cap\ker(\hat\alpha)$. Therefore, $\gamma$ is trivial when viewed as an element in both $\Gamma_0(1)^{\text{ab}}$ and $\widehat\Gamma_0(1)^{\text{ab}}$, so that it is a product of commutators in $\Gamma_0(1)$ and in $\widehat \Gamma_0(1)$.  Using the effective solution to the word problem in $\Gamma_0(1)$ provided by the algorithms of Voight and Page, one can compute this decomposition explicitly. Now using the explicit formula of \cite[Lemma 3.2]{GMS2} one computes elements $z\in \Z[\Gamma_0(1)]\otimes \Z[\Gamma_0(1)]$ and $\hat z\in \Z[\widehat\Gamma_0(1)]\otimes \Z[\widehat\Gamma_0(1)]$ such that $\partial z = \gamma$ and $\partial \hat z = \gamma$. Both $z$ and $\hat z$ can be naturally viewed as elements of $\Z[\Gamma]\otimes \Z[\Gamma]$; then $c=z-\hat z \in \Z[\Gamma]\otimes \Z[\Gamma]$ is a cycle and its class $[c]\in H_2(\Gamma,\Z)$ maps to $\gamma$ under the map $d$. Now let $\tau$ be any element in $\Div^0\cH_\fp$. If we write $c=\sum n_i g_i\otimes h_i$, then by \cite[Lemma 3.3]{GMS} the cocycle $\delta\circ d^{-1}$ is represented by the cycle
\begin{align*}
  \sum n_i h_i\otimes (g_i^{-1}\tau-\tau).
\end{align*}

The computations on cohomology groups are practically the same as for the one-dimensional situation. More precisely, we identify the cohomology group $H^1(\Gamma_0(\fp),\Q)$ as the dual of $H_1(\Gamma_0(\fp),\Q)$. In particular, the basis $\{\gamma_1,\gamma_2\}$ gives rise to a dual basis $\{c_1,c_2\}$, whose elements actually lie in $H^1(\Gamma_0(\fp),\Z)^f$. These correspond, under the isomorphism of Proposition \ref{prop: alpha} to cohomology classes 
\begin{align*}
\omega_1,\omega_2\in H^1(\Gamma,\Meas_0(\P^1(F_\fp),\Z))
\end{align*} 
by means of an explicit formula. The main point is that if $B$ is a ball in $\P^1(F_\fp)$, then either $B$ or $\P^1(F_\fp)\setminus B$ is of the form $g\cO_{F_{\fp}}$ for some $g\in \Gamma/\Gamma_0(\fp)$. Thus a measure in $\Meas_0(\P^1(F_\fp),\Z)$ can be regarded as an element in $\mathrm{Coind}_{\Gamma_0(\fp)}^\Gamma \Z$, satisfying certain additional properties. Now the Shapiro isomorphism gives
\begin{align*}
  H_1(\Gamma_0(\fp),\Z)\simeq H^1(\Gamma,\mathrm{Coind}_{\Gamma_0(\fp)}^\Gamma\Z).
\end{align*}
This isomorphism can be given at the level of cocycles by an explicit formula, as follows. Let $\{\gamma_e\}$ be a system of representatives of $\Gamma/\Gamma_0(\fp)$; for $h\in \Gamma$ denote by $\gamma_{e(h)}$ the representative such that $h\in \gamma_{e(h)}\cdot \Gamma_0(\fp)$. Then, for a cocycle $c\colon \Gamma_0(\fp)\ra\Z$ its image under the above isomorphism is represented by the cocycle $g\mapsto \omega_g$, where $\omega_g (h)=c(r)$ for the unique $r\in \Gamma_0(\fp)$ such that 
\begin{align*}
  \gamma_{e(h)}g=r\gamma_{e'},\ \ \text{ for some representative } \gamma_{e'}. 
\end{align*}
The element $\omega_g$ in principle belongs to $\mathrm{Coind}_{\Gamma_0(\fp)}^\Gamma\Z$. With an appropriate choice of the representatives $\{\gamma_e\}$ introduced in \cite{LRV} (see also \cite[\S 3.3]{GMS2}), one sees that $\omega_g$ in fact belongs to $\Meas_0(\P^1(F_\fp),\Z)$ for all $g\in\Gamma$. In this way we compute the cocycles $\omega_1$ and $\omega_2$.

The last step is to compute the integration pairing. That is to say, for $\gamma\in\{\gamma_1,\gamma_2\}$ and $\omega\in\{\omega_1,\omega_2\}$ we need to compute (an approximation) of $\Xint\times_{\delta\circ d^{-1}\gamma} \rho^{-1}(\omega)$. This presents no difference with the case of dimension $1$, and we use exactly the same algorithm presented in \cite[\S 3.4]{GMS2}.

\subsection{$p$-adic periods of genus two Mumford curves}\label{subsection: p-adic periods of genus two curves}  In \S\ref{subsection: a numerical verification} we will compute the period lattice attached to a modular form $f$. In order to numerically test Conjecture \ref{conj: main}, we need a method to compute the period lattice of the abelian surface $A_f$. In our examples, $A_f$ can be taken to be the Jacobian of a genus two curve $C_f$.  Then we can use Teitelbaum's formulas \cite{Tei}, which provide formulas for the $p$-adic periods of a genus two Mumford curve in terms of the coefficients of an equation of the curve. The strategy of \cite{Tei} is to first express the coefficients of an equation as power series in the $p$-adic periods of the curve, and then to invert these series. In a sense, this is similar to the case of Tate elliptic curves, whose $j$-invariant in terms of the Tate period is given by 
\begin{align*}
  j = \frac{1}{q} + 744 + 196884q + \cdots,
\end{align*}
and the expression for the Tate period in terms of the $j$-invariant in obtained by inverting the series. The formulas in the genus two case are rather more complicated, but since they are a key tool in our computations we give an account on how they are obtained (and we also take the chance to correct some minor typos in the formulas of \cite{Tei}). 

Let $F$ be a complete local field. Denote by $\cO_F$ its ring of integers, $\pi$ a uniformizer, and $k$ the residue field, which we assume to be of characteristic $\neq 2$. Let $X$ be a genus two Mumford curve over $F$; that is, $X$ is smooth, irreducible, projective, and it has a stable model over $\cO_F$ such that all components of its generic fiber are isomorphic to $\P^1_k$ and all double points are $k$-rational. Then Mumford's uniformization theory guarantees the existence a free subgroup of rank two $\Gamma_X\subset \PGL_2(F)$ such that $X^{an}\simeq \Omega/\Gamma_X$ (here $\Omega\subset \P^1_F$ is the set of non-limit points of $\Gamma_X$ and $X^{an}$ denotes the rigid analytic space attached to $X$). Manin and Drinfeld \cite{Manin-Drinfeld} constructed a symmetric pairing 
\begin{align}\label{eq: manin pairing}
  \langle \cdot ,\cdot \rangle \colon \Gamma_X \times \Gamma_X  \lra F^\times
\end{align}
that uniformizes the Jacobian of $X$. That is, \eqref{eq: manin pairing} embeds $\Gamma_X$ as a discrete subgroup of $\Hom(\Gamma_X,F^\times)$ and, if we call $\Lambda$ the image of $\Gamma_X$, then the quotient  $\Hom(\Gamma_X,F^\times)/\Lambda$ is isomorphic to $\Jac(X)$. 

Let us assume that the Weierstrass points of $X$ are defined over $F$. The reduction of $X$ determines a partition of the set of Weierstrass points into three subsets $S_1, S_2,S_3$ of two elements, which can be read off from the distribution of the Weierstrass points in the special fiber of a minimal regular model for $X$ \cite[Proposition 9]{Tei}. To be more precise, the set of Weierstrass points reduce to either three points $w_1$, $w_2$, $w_3$ of multiplicities $2$, $2$, $2$, or to four points $w_1$, $w_2$, $w_3$, $w_3'$ with multiplicities $2$, $2$, $1$, $1$. For $i=1,2$ the set $S_i$ consists of the points that reduce to $w_i$. In the first case, $S_3$ consists on the points that reduce to $w_3$, whereas in the second case it consists on those points that reduce to $w_3$ or $w_3'$. Observe that the ordering of the $S_i$'s is not uniquely determined, although in the second case the set $S_3$ is distinguished.

For each $i$ we choose a labeling of the points in $S_i$ as $S_i = \{ P_i^+,P_i^-\}$. Let $\gamma_1,\gamma_2,\gamma_3\in \Gamma_X$ be the elements associated to the choice  of this labeling as in \cite[\S 2.1]{Tei}. For our purposes, we do not need to know much about these elements, just that they generate $\Gamma_X$ and that $\gamma_1\gamma_2\gamma_3 = 1$. Teitelbaum defines the fundamental periods of $X$ as
\begin{align*}
  q_1 = \langle \gamma_2,\gamma_3\rangle^{-1},\   q_2 = \langle \gamma_1,\gamma_3\rangle^{-1},\   q_3 = \langle \gamma_1,\gamma_2\rangle^{-1}.
\end{align*}
Observe that a period lattice of $\Jac(X)$ is given by the columns of the matrix 
\begin{align}
\label{eq:teitelbaum-periods}
  \mtx A B B D := \mtx{\langle \gamma_1, \gamma_1 \rangle}{\langle \gamma_1, \gamma_2 \rangle}{\langle \gamma_2, \gamma_1 \rangle}{\langle \gamma_2, \gamma_2 \rangle},
\end{align}
so indeed the lattice $\Lambda$ can be recovered from the $q_i$'s. For later use, we note that the periods can also be recovered from the lattice:
\begin{align}\label{eq: qs in terms of ABC}
q_1 = BD, \quad q_2 = AB,\quad q_3 = B^{-1}.
\end{align}

 A related notion, which plays a key role in Teitelbaum's formulas, are the so-called half-periods of $X$, defined as
\begin{align*}
  p_1 = \chi_{12}(\gamma_2), \  p_2 = \chi_{23}(\gamma_3),\  p_3 = \chi_{31}(\gamma_1),
\end{align*}
where $\chi_{12}$, $\chi_{23}$, and $\chi_{31}$ are the elements in $\Hom(\Gamma_X,F^\times)$ defined in \cite[Definition 17]{Tei}. The $\chi_{ij}$'s are defined by very explicit expressions, but we do not need them actually. Indeed, one of the main results of \cite{Tei} is an expression of the coefficients of an equation of $X$ as power series in the $p_i$'s; by inverting these series one obtains formulas for the $p_i$'s in terms of an equation of $X$. These are the formulas that we are interested in, since the $p_i$'s are related to the $q_i$'s by the simple relation $q_i = p_i^{-2}$ (which also justifies the name of half periods).

Before stating the formulas, we still need to introduce some more notation. Let 
\begin{align}
(\cdot,\cdot)\colon \Gamma_X^{ab}\times \Gamma_X^{ab}\lra F^\times
\end{align} 
be the bilinear pairing defined by $(\gamma_i,\gamma_j)=p_k^{-1}$ for different $i,j,k$ (observe that this determines the pairing completely because of the relation $\gamma_1\gamma_2\gamma_3=1$), and let $\theta\colon \Hom(\Gamma_X,F^\times)\lra F^\times$ be the theta function defined by the formula
\begin{align*}
  \theta (\chi)= \sum_{\gamma\in\Gamma_X^{ab}} (\gamma,\gamma)\chi(\gamma).
\end{align*}
 Table 1  of \cite{Tei}  defines 15 characters, denoted as $\chi_{P,Q}$ for various choices of $$P,Q\in \{P_1^+, P_1^- ,P_2^+, P_2^-,P_3^+, P_3^-\},$$ by means of explicit formulas in terms of the half periods. These characters, together with the trivial character $\chi_1$, are the set of $2$-torsion points of the Jacobian of $X$ (under the identification $\Jac(X)\simeq \Hom(\Gamma_X,F^\times)/\Lambda$).

Let $x$ be the function on $X$ that has a double pole at $P_1^+$, a double zero at $P_2^+$, and such that $x(P_3^+)=1$. Then, possibly after  adding the square root of some element of $F$, we can assume that $X$ has a model given by
\begin{align*}
  y^2 = x (x-1)(x-x(P_1^-))(x-x(P_2^-))(x-x(P_3^-)).
\end{align*}
Then by \cite[Theorem 28]{Tei} we have that
\begin{align}
\label{eq:xsandps}  x(P_3^-)&= \frac{\theta^2(\chi_{P_2^+,P_2^-})\theta^2(\chi_{P_1^-,P_2^+})}{\theta^2(\chi_{P_1^+,P_1^-})\theta^2(\chi_{P_1^+,P_2^-})},\\\nonumber
 \frac{x(P_1^-)- 1}{x(P_1^-)}&= \frac{\theta^2(\chi_{P_3^+,P_3^-})\theta^2(\chi_{P_2^-,P_3^+})}{\theta^2(\chi_{P_2^+,P_2^-})\theta^2(\chi_{P_2^+,P_3^-})},\\\nonumber
 \frac{1}{1-x(P_2^-)}&= \frac{\theta^2(\chi_{P_1^+,P_1^-})\theta^2(\chi_{P_3^-,P_1^+})}{\theta^2(\chi_{P_3^+,P_3^-})\theta^2(\chi_{P_3^+,P_1^-})}.
\end{align}
Note that there is a typo in display $(23)$ of \cite{Tei}, in which the left hand sides of the last two equations are swapped. These formulas express $x(P_1^-)$, $x(P_2^-)$, and $x(P_3^-)$ in terms of the half periods $p_1$, $p_2$, and $p_3$, which can be calculated using the definitions of the characters $\chi_{P,Q}$ and of $\theta$; explicitly:
\begin{align*}
  \theta(\chi_{P_1^+,P_1^-}) &= \sum_{i,j\in\Z} p_2^{i^2}p_1^{j^2}p_3^{(i-j)^2}(-1)^j,&
\theta(\chi_{P_2^+,P_2^-}) &= \sum_{i,j\in\Z} p_2^{i^2}p_1^{j^2}p_3^{(i-j)^2}(-1)^i,\\
\theta(\chi_{P_3^+,P_3^-}) &= \sum_{i,j\in\Z} p_2^{i^2}p_1^{j^2}p_3^{(i-j)^2}(-1)^{i+j},&
\theta(\chi_{P_1^-,P_2^+}) &= \sum_{i,j\in\Z} p_2^{i^2-i}p_1^{j^2-j}p_3^{(i-j)^2}(-1)^{i+j},\\
\theta(\chi_{P_1^+,P_2^-}) &= \sum_{i,j\in\Z} p_2^{i^2-i}p_1^{j^2-j}p_3^{(i-j)^2},&
\theta(\chi_{P_2^-,P_3^+}) &= \sum_{i,j\in\Z} p_2^{i^2+i}p_1^{j^2}p_3^{(i-j)^2+(i-j)}(-1)^j,\\
\theta(\chi_{P_2^+,P_3^-}) &= \sum_{i,j\in\Z} p_2^{i^2+i}p_1^{j^2}p_3^{(i-j)^2+(i-j)},&
\theta(\chi_{P_3^-,P_1^+}) &= \sum_{i,j\in\Z} p_2^{i^2}p_1^{j^2+j}p_3^{(i-j)^2-(i-j)}(-1)^i,\\
\theta(\chi_{P_3^+,P_1^-}) &= \sum_{i,j\in\Z} p_2^{i^2}p_1^{j^2+j}p_3^{(i-j)^2-(i-j)}.\\
\end{align*}

Note that there is a typo in display $(25)$ of \cite{Tei} in the sign affecting the sum. In this way, these expressions give the coordinates $x(P_1^-)$, $x(P_2^-)$, and $x(P_3^-)$ as power series in the half periods $p_1$, $p_2$, and $p_3$. 

If one starts with a genus two Mumford curve $X/F$, one can compute its half periods $p_i$, and therefore its periods $q_i=p_i^{-2}$, as follows. Consider a hyperelliptic model $y^2 = f(x)$, with $f(x)$ of degree $6$. After possibly changing $F$ we can assume that $f$ has its six roots in $F$, and hence the Weierstrass points are defined over $F$; the next step is to label them as $P_1^+, P_1^-,P_2^+,P_2^-,P_3^+,P_3^-$. The reduction $\tilde f$ of $f$ $\pmod \pi$ factors either as 
\begin{align*}
  \tilde f (x) = (x-\tilde x_1)^2 (x-\tilde x_2)^2 (x-\tilde x_3)^2,
\end{align*}
or as
\begin{align*}
  \tilde f (x) = (x-\tilde x_1)^2 (x-\tilde x_2)^2 (x-\tilde x_3)(x-\tilde x_3').
\end{align*}
In the first case we take $x(P_i^{\pm})$ as the roots reducing to $\tilde x_i$ (in any order). In the second, $x(P_1^{\pm})$ (resp. $x(P_2^{\pm})$) reduce to $\tilde x_1$ (resp. $\tilde x_2$), and $x(P_3^\pm)$ are the points reducing to $\tilde x_3$ and $\tilde x_3'$ (again, in any order). Next, apply the transformation
\begin{align*}
x\mapsto  \frac{(x-x(P_2^+))(x(P_3^+)-x(P_1^+))}{(x-x(P_1^+))(x(P_3^+)-x(P_2^+))},
\end{align*}
which sends $P_1^+$ to $\infty$, $P_2^+$ to $0$, and $P_3^+$ to $1$; denote again by $x(P_1^-)$, $x(P_2^-)$, and $x(P_3^-)$ the $x$-coordinates of the $P_i^-$ after applying the transformation. Then the half periods are obtained by plugging these values in \eqref{eq:xsandps}, and solving for $p_1$, $p_2$, and $p_3$. The power series can be inverted, as is done in \cite[p. 141]{Tei}. However, computationally it is more efficient to solve the system numerically, by applying a three-dimensional Newton scheme.

\subsection{A numerical verification}\label{subsection: a numerical verification} Consider the cubic field $F=\Q(\alpha)$, where $\alpha$ is a root of the polynomial $x^{3} - x^{2} + 3 x - 2 $. This field has signature $(1,1)$. Let $\fp = (\alpha^2+1)$, which has norm $p=5$, and let $\fd = (-2\alpha^2 + 4\alpha - 7)$, of norm $173$. We see that in this example $F_\fp=\Q_p=\Q_5$.

Let $B$ be the (unique) quaternion algebra over $F$ of discriminant $\fd$. It is generated by $i$ and $j$ satisfying $i^2=-4\alpha^2+4\alpha-7$ and $j^2=-173$. We have computed the cohomology group $H^1(\Gamma_0(\fp),\C)$ associated to norm-one units of an Eichler order of level $\fp$ in $B$, and found it to be of dimension $4$. The Hecke operator $T_\fl$, where $\fl=(\alpha)$, acts on it with characteristic polynomial $(x^2+x-4)(x^2+x-1)$. Therefore there are two irreducible $2$-dimensional components. We consider the component corresponding to the factor $x^2+x-1$, which corresponds to a newform $f$ whose eigenvalues generate the quadratic field $\Q(\sqrt{5})$.

Using the methods described in \S\ref{Algorithms for computing the $p$-adic lattice} (the actual code for the implementation can be found in \texttt{https://github.com/mmasdeu/darmonpoints}) we find a basis $\{\varphi_1,\varphi_2\}$ of $H^1(\Gamma_0(\fp),\Z)^f$. Similarly, we find a basis $\{\theta_1',\theta_2'\}$ of $H_1(\Gamma_0(\fp),\Z)^f$. If we identify $H_1(\Gamma_0(\fp),\Q)^f$  with the dual of $H^1(\Gamma_0(\fp),\Q)^f$, these two bases are not dual to each other: the matrix of cap-product pairings is
\[
\left(\begin{array}{rr}
-142456600326 & -18 \\
94971066876 & 12
\end{array}\right),
\]
which has determinant $-144$. This leads us to consider a new basis for the homology which is a ``pseudo''-dual basis for $\{\varphi_1,\varphi_2\}$ and which has the property that the resulting cap-product is
\[
\left(\begin{array}{rr}
-144 & 0 \\
0 & -144
\end{array}\right).
\]
On this new basis $\{\theta_1,\theta_2\}$, the Hecke operator $T_{\fl}$ acts as
\[
\mat{T_{\fl}\theta_1\\T_{\fl}\theta_2}=\mtx{-2}{1}{2}{1}^t \mat{\theta_1\\\theta_2}.
\]
Computing the integration pairing \eqref{eq: pairing 2} we obtain the periods:
\begin{align*}
A_0 &= \langle \varphi_1,\theta_1 \rangle = 5^{45}\cdot 227015308497264163898130173143 \pmod{5^{87}},\\
B_0 &= \langle \varphi_1,\theta_2 \rangle =  30930079006020210124765717907 \pmod{5^{42}}.\\
\end{align*}
The periods $C_0=\langle \varphi_2,\theta_1\rangle$ and $D_0=\langle \varphi_2,\theta_2\rangle$ are readily obtained from $A_0$, $B_0$ and the matrix $T_{\fl}$, by Hecke-equivariance of the pairing. In this particular case, for example:
\begin{align*}
C_0 &= B_0^2 & D_0 &= A_0 B_0^{3}.
\end{align*}
The lattice $\Lambda_f'$ is therefore generated by the columns of the matrix $\smtx{A_0}{B_0}{C_0}{D_0}$.

Now, we want to verify that $\Lambda_f'$ is isogenous to the period lattice of $A_f$. For this, consider the genus two curve $X/F$ given by the following equation:
\begin{align*}
 y^2 + &(x^3 + (-\alpha - 1)x^2 - \alpha x + 1)y =\\
&(2\alpha^2 - 4\alpha + 2)x^4 + (4\alpha^2 - 8\alpha + 3)x^3 + (5\alpha^2 - 7\alpha + 3)x^2 + (3\alpha^2 - 3\alpha + 1)x + \alpha^2 - \alpha.
\end{align*}
We have obtained this curve by specializing the parameters in the Brumer family (see \cite{hashimoto}), and therefore the endomorphism algebra of $\Jac(X)$ contains $\Q(\sqrt{5})$. The conductor of $\Jac(X)$ can be computed using the Magma: it is precisely $\fp^2\fd^2$. Moreover, we have computed the eigenvalue $a_\fl(f)$ and we have checked that the relation
\[
\# A_f(\cO_F/\fl)= N_{\Q(\sqrt{5})/\Q}(1+|\fl| - a_\fl(f))
\]
holds. These properties lead us to think that, in all likelihood, $\Jac(X)$ is isogenous to $A_f$. Since we have the equation of $X$, we have been able to use Teitelbaum's formulas recalled in Section \ref{subsection: p-adic periods of genus two curves} to compute an approximation (up to $50$ digits) to the $p$-adic periods $A$, $B$, $D$ of $\Jac(X)$ as in \eqref{eq:teitelbaum-periods}. This is precisely the lattice $\Lambda_f$.

Now, in order to check that $\Lambda_f$ and $\Lambda_f'$ are isogenous, as predicted by Conjecture \ref{conj: main}, we need to find matrices $Y,Z\in \M_2(\Z)$ such that $V^Y = {}^Z W$, where $V=\smtx A B B D$ and $W = \smtx{A_0}{B_0}{C_0}{D_0}$. What we do, in fact, is to find matrices satisfying the weaker relation (see \eqref{eq: logs} for the notation):
\begin{align}
\label{eq:kadziela-relation}
Y \ell(V) = \ell(W) Z,
\end{align}
where $\ell$ denotes a $\fp$-adic logarithm. To this purpose, lattice reduction techniques allow us to express the logarithms of the Teitelbaum periods $A$, $B$ and $D$ that we computed for $X$ in terms of small linear combinations of the integration periods $A_0$, $B_0$. In this particular example, the following relation holds (up to the working precision):
\[
\mat{\log A\\\log B\\\log D} = \mat{9&-6\\-6&3\\9&-3} \mat{\log A_0\\\log B_0}.
\]
 Since $A_0$ and $B_0$ are assumed to be algebraically independent, \eqref{eq:kadziela-relation} yields a homogeneous system of eight equations in eight variables (the entries of $Y$ and $Z$). It turns out that this system has a two-dimensional space of solutions, which allows us to extract the sought (non-zero) matrices $Y$ and $Z$. For example, one can take the Kadziela matrices
\[
Y = \mtx 1{4}{0}{-1}\quad Z=\mtx{-15}{30}{6}{-9}.
\]
A posteriori we check that the multiplicative relation $V^Y = {}^Z W$ holds as well, thus giving  numerical evidence that indeed the lattice $\Lambda_f'$ is isogenous to the lattice of $A_f$.

\section{Equations of genus two curves}\label{section 5}
In the previous section we have computed the period lattice $\Lambda_f'$ and, in order to compute the periods of $A_f$ and check Conjecture \ref{conj: main}, we have crucially used that in that case $A_f$ is the Jacobian of a genus two curve $X/F$. We remark that the equation of $X$ was known beforehand.

In this section, we describe a method to recover a model for the curve $X$ directly from the lattice $\Lambda_f'$. Since $\Lambda_f'$ is computed using only information of the modular form $f$, this can be seen as the higher dimensional analog of the method to compute the equation of the elliptic curve $A_f$ when $f$ has rational Hecke eigenvalues that was introduced in \cite{GMS}.

The idea of the method is as follows. First of all we compute the matrix $W=\smtx{A_0}{B_0}{C_0}{D_0}$ whose columns are the basis of $\Lambda_f'$ by means of the integration pairing. Suppose that $A_f$ is isogenous to the Jacobian of a genus two curve $X/F$, and let $V=\smtx{A}{B}{C}{D}$ be the period matrix of $X$. Granting Conjecture \ref{conj: main}, these two matrices will be related by
\begin{align*}
  V ^Y = {}^Z W,
\end{align*}
for some matrices $Y,Z\in \M_2(\Z)$. Assuming for now that we can find $Y$ and $Z$ (we will say later on how to do it), we can solve for $V$, and therefore obtain the periods $A$, $B$, $D$ of the curve $X$. From this, one readily recovers the periods $q_1$, $q_2$, $q_3$ using the relations \eqref{eq: qs in terms of ABC}:
\[
q_1 = BD, \quad q_2 = AB,\quad q_3 = B^{-1}.
\]
Taking square roots one recovers the half periods $p_i=q_i^{-1/2}$, which can be plugged in the power series for the Theta series~\eqref{eq:xsandps} and ultimately give the Weierstrass points of $X$. The resulting approximations to the Weierstrass points, although algebraic, are not defined over $F$ in general and we cannot hope to recognize them algebraically from these $p$-adic approximations. Instead, what we can to is to compute a $p$-adic approximation of the (absolute) invariants of the curve $X$, which are defined as
\begin{align*}
  i_1 = \frac{I_2^5}{I_{10}},\ \ i_2 = \frac{I_2^3I_4}{I_{10}}, \ \ i_3 = \frac{I_2^2I_6}{I_{10}}.
\end{align*}
Here, $I_2,I_4,I_6,I_{10}$ are the Igusa--Clebsch invariants, defined as certain symmetric polynomials in the Weierstrass points of $X$. Since $X$ is defined over $F$, then $i_1,i_2,i_3$ do belong to $F$.

Now, from the approximations to the Weierstrass points we can compute approximations to the invariants, say $\tilde i_1,\tilde i_2,\tilde i_3$. If we have enough precision we will be able to recognize them as elements of $F$. This gives us the method for finding the matrices $Y$ and $Z$: we run over pairs of $2\times 2$ integer matrices and, for each trial, we compute the resulting approximation to the invariants $\tilde i_1,\tilde i_2,\tilde i_3$; if we are not able to identify them as elements of $F$, this likely means that $Y$ and $Z$ are wrong and we try with a different pair.

Trying over pairs of matrices $Y,Z$ a priori involves exploring an eight-dimensional lattice, but one can reduce the search by imposing that the resulting matrix $V$ is symmetrical. The finding of $Y$ and $Z$, which is the two-dimensional generalization of the problem of guessing the valuation of the Tate parameter of an elliptic curve from its $p$-adic $L$-invariant, is one of the places that makes the computation challenging.

In addition to this, this strategy soon reveals another problem: the absolute invariants have very large height compared to the coefficients of a minimal Weierstrass model. If for a given pair $Y,Z$ we are not able to recognize $\tilde i_1,\tilde i_2,\tilde i_3$ as elements in $F$, this could also be because the height of the invariants $i_1,i_2,i_3$ is too large to be recognized with our working precision, and it would prevent the method to work unless one was able to compute to extremely high precision.

Instead, there is an improvement to the method that makes the computation feasible in some cases: since we are aiming at curves with a specific conductor (obtained from the level of the modular form from which we started), we have a certain control on what the discriminant of the curve should be. Unlike the case of elliptic curves, the support of the discriminant of the genus two curve may be larger than that of its conductor, and hence it is in principle possible that there are more primes in the discriminant than those appearing in our level. We expect this not to happen except for small primes, which we can include in our search. Suppose that we guess the discriminant $I_{10}$ of the curve. Then from the approximations $(\tilde i_1,\tilde i_2,\tilde i_3)$ to the invariants we can extract approximations to $I_2=(i_1I_{10})^{1/5}$, $I_4 = i_2 I_{10}I_2^{-3}$ and $I_6 = i_3 I_{10} I_2^{-2}$, which are defined over $F$ (because $I_{10}$ belongs to $F$) and have lower height than $I_{10}$ itself.

As a proof of concept, we illustrate the method with an example for which we had the target curve beforehand. Consider the number field $F=\Q(\alpha)$, with $\alpha$ a root of the polynomial $x^3-x^2+1$, which has signature $(1,1)$. We consider the quaternion algebra with relations $i^2=9\alpha^2-3\alpha-11$, $j^2=-2\alpha^2$, which has discriminant $\fd=(8\alpha^2-10\alpha-1)$ of norm $821$. Let $\fp=(-2\alpha^2+\alpha)$ be the unique prime of norm $7$ in $F$. In this case we find a $2$-dimensional component of the cohomology and homology, on which $T_\fl$ (with $\fl = (\alpha^2+\alpha-2)$, a prime of norm $11$) acts with characteristic polynomial
$x^2-2x-19$. In fact, with respect to the chosen pseudo-orthogonal bases $\{\varphi_1,\varphi_2\}$, $\{\theta_1,\theta_2\}$, the operator $T_\fl$ acts on the homology via the matrix
\[
\mat{T_\fl\theta_1\\T_\fl\theta_2}= \mtx{-1}{-4}{-4}{3}^t \mat{\theta_1\\\theta_2}.
\]
This corresponds then to a newform $f$ such that $K_f\simeq \Q(\sqrt{5})$, and we aim to compute the invariants of a curve $X/F$ such that $\Jac(X)=A_f$.

The first step is to compute the periods of the lattice $\Lambda_f'$ by means of the integration pairing, as in the previous section. They are:
\begin{align*}
  A_0 &= 7^{-4}\cdot 27132321333884163473566078077966608077268973477\pmod{7^{52}}\\
  B_0 &= 397745278075295216478310410412961033205591801491513\pmod{7^{60}}.
\end{align*}

Guessing the Kadziela matrices
\[
Y = \mtx{-1}{-1}{-1}{0} \quad Z=\mtx{1}{1}{1}{0},
\]
which can be done by looping over matrices with small entries, we can compute a new set of periods
\begin{align*}
A &= 7^{-1} \cdot 180373636240760651045145390062543188665673147874 + O(7^{55})\\
B &= 101858856942719452845868815022429183828273612324 + O(7^{56})\\
D &= 7^{-1} \cdot 80209973804903028832117210143467211207304220322 + O(7^{55})
\end{align*}

This yields $q_1$, $q_2$, $q_3$:
\begin{align*}
q_1 &=7^{-1} \cdot 180373636240760651045145390062543188665673147874 + O(7^{55})\\
q_2 &=7^{-1} \cdot 146582506580515644910043665072399073999059180487 + O(7^{55})\\
q_3 &=2524063863085285102995202849415046621669591961 + O(7^{56}).
\end{align*}
From this we compute the half periods taking square roots and the Weierstrass points using formula \eqref{eq:xsandps}. With the Weierstrass points, we can compute approximations to the invariants
\begin{align*}
i_1 = I_2^5 I_{10}^{-1} &= 7^{-2} \cdot 383000380988298534086703050832398358583029537 + O(7^{51})\\
i_2 = I_2^3 I_4 I_{10}^{-1} &= 7^{-2} \cdot 216286438165031483296107998530348655636952080 + O(7^{51})\\
i_3 = I_2^2 I_6 I_{10}^{-1} &=7^{-2} \cdot 17712448343391292208503851621997332642044090 + O(7^{50}).
\end{align*}
The discriminant of the sought hyperelliptic curve should have support $\{2, \fp, \fd\}$. In this case, the fundamental unit of $F$ is $\alpha$, so we have tried discriminants of the form 
\begin{align*}
  I_{10}=\alpha^a2^b(-2\alpha^2+\alpha)^2 (8\alpha^2-10\alpha-1)^2.
\end{align*}
That is, for different pairs of $a$ and $b$ we have computed $I_2$, $I_4$, $I_6$, and we have tried to identify them as elements of $F$. This has worked for $a=-12$ and $b=12$, so that our guess of the discriminant is
\[
I_{10} = \alpha^{-12} 2^{12} (-2\alpha^2+\alpha)^2 (8\alpha^2 -10\alpha -1)^2
\]
With this $I_{10}$, we have recognized $I_2$, $I_4$, and $I_6$ to be the following elements in $F$:
\begin{align*}
I_2 &=576 \alpha^{2} - 712 \alpha + 840,\\
I_4 &=7396 \alpha^{2} - 11208 \alpha + 9636,\\
I_6 &=2882256 \alpha^{2} - 4646648 \alpha + 3543824.
\end{align*}
It is worth remarking how the quantities $I_2$, $I_4$, $I_6$ were actually found. If one has access to arbitrarily high precision, one can simply use the \texttt{algdep} commands that exist in Sage or Pari, which return a polynomial approximately satisfied by the input. However, since we are expecting to find elements in $F$ we may be more successful finding linear dependency relations among the input and a power basis of $F$. In this example, while $50$ $p$-adic digits suffice to recover the Igusa invariants, the \texttt{algdep} command would require instead about $90$ $p$-adic digits, which would make the computation nearly unfeasible with the computational resources  available to us.

Under Conjecture \ref{conj: main}, the invariants $(I_2,I_4,I_6,I_{10})$ should correspond to the invariants of a genus two curve $X/F$ such that $\Jac(X)=A_f$. At this point, and one can run Mestre's algorithm, which is implemented in \texttt{Magma} to find the equation of a curve $X'/F$ with these invariants. The curve $X'$ will be then a twist of $X$, and by looking at its conductor one can untwist it to recover $X$. This is the strategy used, in a similar situation, in \cite[\S 4.1.1]{dembele-kumar}. However, in this case the model that Mestre's algorithm outputs has very large height (as is usually the case), and since $F$ is not totally real we do not currently dispose of algorithms to reduce it to a more manageable size.

For this example, we provide an independent check of the fact that the above invariants seem to correspond to the invariants of a curve whose Jacobian is $A_f$. Consider the curve given by:
{\small
\[
y^2 + (x^3 + (-\alpha^2 - 1)x^2 - \alpha^2x + 1)y =(-\alpha ^2 + 1)x^4 - 2\alpha ^2x^3 + (-\alpha ^2 - 3\alpha  - 1)x^2 + (-3\alpha  - 2)x - \alpha  - 1.
\]
}
As in the previous section, we have obtained it by specialization of the parameters in the Brumer family, so the endomorphism algebra of its Jacobian contains $K_f=\Q(\sqrt{5})$. Moreover, it has conductor $\fp^2\fd^2$, and we have checked that the number of points modulo $\fl$ that lie on the curve agrees with the expected value according to the eigenvalue of $f$. It seems to correspond, then, to the abelian surface $A_f$, and one can check that its Igusa--Clebsch invariants coincide with the ones found above.  This gives evidence that  method presented in this section can be used to compute the Igusa invariants of a genus two curve whose Jacobian is $A_f$ in terms of the integration pairing $\langle\cdot  ,\cdot \rangle_f$.

\section{$p$-adic $L$-invariants}\label{section 6}
The Hecke-equivariant pairing $\langle \cdot , \cdot \rangle_f$ of Equation~\eqref{eq: pairing 2}, together with Conjecture \ref{conj: main}, allows for the definition of a $p$-adic $L$-invariant of $f$. In the case of dimension $2$, it can be computed explicitly by means of the methods of Section \ref{section 4}. The fact that it coincides when $F=\Q$ with the $p$-adic $L$-invariant introduced in \cite{MTT} allows for an alternative method of computation as well.

We define the $p$-adic $L$-invariant only for abelian surfaces, the higher dimensional case is analogous. Let $\mathbb{T}_f\subset \End(H^{1}(\Gamma_0(\fp\fm),\Q)^f)$ be the set of endomorphisms generated by the Hecke operators acting on $H^{1}(\Gamma_0(\fp\fm),\Q)^f$. We have that $\mathbb{T}_f\otimes \Q$ is a number field isomorphic to $K_f$, which also acts on $H_{1}(\Gamma_0(\fp\fm),\Q)^f$ in a way compatible with $\langle \cdot , \cdot \rangle_f$. Moreover, $H_{1}(\Gamma_0(\fp\fm),\Q)^f$ and $H^{1}(\Gamma_0(\fp\fm),\Q)^f$ are $K_f$-modules of rank $1$. 

Now consider the maps $\alpha,\beta \colon H_{1}(\Gamma_0(\fp\fm),\Q)^f\ra \Hom (H^1(\Gamma_0(\fp\fm),\Q)^f,F_\fp)$ given by
\begin{align*}
 \alpha(\theta)(\varphi) = \ord_\fp(\langle \theta,\varphi\rangle_f), \ \  \beta(\theta)(\varphi) = \log_\fp(\langle \theta,\varphi\rangle_f),
\end{align*}
where $\ord_\fp$ denotes the order at $\fp$ and $\log_\fp$ a $\fp$-adic logarithm (say, the one that takes a generator of $\fp$ to $0$). Under Conjecture \ref{conj: main} the map $\ord_\fp \langle \cdot , \cdot \rangle_f$ is non-degenerate, making $\alpha$ an isomorphism. The $p$-adic $L$-invariant is defined as the unique element $\mathcal{L}_\fp(f)\in \mathbb{T}_f\otimes F_\fp$ such that $\beta = \mathcal{L}_\fp(f) \alpha$. 

\subsection{Examples over $\Q$} Teitelbaum introduced in \cite{Tei} a method to compute $p$-adic $L$-invariants of Jacobians of genus two modular curves. In this section we compute the $p$-adic $L$-invariants attached to $2$-dimensional simple factors of Jacobians of modular curves of genus larger than $2$, for which in principle one cannot use the methods of~\cite{Tei}.

As a first example, consider the space of modular forms of level $\Gamma_0(165)$ (note that $165= 3\cdot 5\cdot 11$). We work with the indefinite quaternion algebra $B$ of discriminant $15$ and with $p=11$. We find a two-dimensional factor of the cohomology $H^1(\Gamma_0^B(11),\Z)$, on which the Hecke operator $T_2$ acts with characteristic polynomial $x^2+2x-1$. This corresponds to a modular form $f$ whose Hecke eigenvalues generate the field $\Q(\sqrt{2})$. This defines an isomorphism $\mathbb{T}_f\otimes_{\Z}\Q_{p} \cong \Q_{p^2}=\Q_{p}(T_2)$, and we find that $\mathcal{L}_{p}(f)$ is
\[
11 \cdot 2434708053353386815382354389779 + 11^2 \cdot 1134757179957513984261268713424\cdot T_2 + O(11^{31}).
\]

As a second example, consider this time $\Gamma_0(357)$ (note that $357=3\cdot 7\cdot 17$). We work with the quaternion algebra of discriminant $3\cdot 17$, and with $p=7$. In this case, the Hecke algebra is generated by $T_5$ which has characteristic polynomial $x^2+2x-1$ (the fact that it is the same as the previous example is a mere coincidence), giving an isomorphism $\mathbb{T}_f\otimes_{\Z} \Q_p\cong\Q_{p^2}=\Q_p(T_5)$, and in this case we find that $\mathcal{L}_p(f)$ is
\[
7 \cdot 15066781074161344457224002 + 7^2 \cdot 1814973922464853030271319\cdot T_5 + O(7^{31}).
\]

\subsection{Examples over cubic number fields}
We have also computed $p$-adic $L$-invariants of some modular forms over number fields $\Q(\alpha)$. Here is the summarized data for the modular forms of the examples in the previous sections.

The example of \S\ref{subsection: a numerical verification}, with minimal polynomial for $\alpha$ being: $ x^3 - x^2 + 3x - 2$:
\begin{itemize}
\item Level $\fp\cdot \fd$:  $(\alpha^2 + 1)_5 \cdot (-2\alpha^2 + 4\alpha - 7)_{173}$.
\item Characteristic polynomial of $T$: $x^{2} + x - 1$.
\item $\mathcal{L}_5(f) = 7483779755785384529304478059 + 1668041363337346469653221493\cdot T + O(5^{40})$.
\end{itemize}

The example of \S\ref{section 5}, with minimal polynomial for $\alpha$ being $x^3 - x^2 + 1$:
\begin{itemize}
\item Level $\fp\cdot\fd$:  $(-2\alpha^2 + x)_7\cdot(8\alpha^2 - 10\alpha - 1)_{821}$.
\item Characteristic polynomial of $T$: $x^{2} - 2x - 19$.
\item $\mathcal{L}_7(f) =  7 \cdot 20049683766104040108804775 + 7 \cdot 10498143651203088689572467 \cdot T+ O(7^{31})$.
\end{itemize}

\bibliographystyle{amsalpha}
\bibliography{refs}

\end{document}